\documentclass{birkjour}
\newtheorem{thm}{Theorem}[section]
 
 \newtheorem{lem}[thm]{Lemma}
 \newtheorem{prop}[thm]{Proposition}
 \theoremstyle{definition}
 
 \theoremstyle{remark}
 \newtheorem{rem}[thm]{Remark}
 \newtheorem*{claim}{Claim}
 
 \numberwithin{equation}{section}
 
\usepackage{verbatim}
\usepackage{amsmath}
\usepackage{amssymb}
\usepackage{amsthm}
\usepackage{cancel}

\newtheorem{hypothesis}[thm]{Hypothesis}
\newcommand*{\bigchi}{\mbox{\large$\chi$}}
\DeclareMathOperator\erf{erf}

\usepackage{hyperref}

\newcommand{\cB}{{\mathcal B}}

\newcommand{\cI}{{\mathcal I}}

\newcommand{\cK}{{\mathcal K}}

\newcommand{\sH}{{\mathsf H}}

\newcommand{\Tr}{{\operatorname{Tr}}}

\begin{document}

%
%
%
%
%
%
%
%
%

\title[Numerical Fredholm determinants]
 {Numerical Fredholm determinants for matrix-valued kernels on the real line}

\author[Gallo]{Erika Gallo}

\address{%
Department of Mathematics\\
The University of Texas at Arlington\\\
Arlington, TX, 76019\\
USA}

\email{erika.gallo@uta.edu}

\author[Zweck]{John Zweck}
\address{%
Department of Mathematics\\
New York Institute of Technology\\
New York, NY, 10023\\
USA\\
\href{https://orcid.org/0000-0001-9255-9757}{ORCID: 0000-0001-9255-9757}}
\email{jzweck@nyit.edu}

\author[Latushkin]{Yuri Latushkin}
\address{%
Department of Mathematics\\
University of Missouri\\
Columbia, MO 65211\\
USA\\
\href{https://orcid.org/0000-0002-8259-5655}{ORCID: 0000-0002-8259-5655}}
\email{latushkiny@missouri.edu}




\begin{abstract}
We analyze a numerical method for computing Fredholm determinants of 
trace class and Hilbert Schmidt integral operators defined in terms of matrix-valued 
kernels on the entire real line. With this method, the Fredholm determinant
is approximated by the determinant of a matrix constructed by truncating
the  kernel of the operator to a  finite interval and then applying a quadrature rule. 
Under the assumption that
the kernel decays exponentially, we  
derive an estimate relating the Fredholm determinant of
the operator on the real line to that of its truncation to a finite interval.
Then we derive a quadrature error estimate relating the 
Fredholm determinant of a matrix-valued kernel on a finite interval to
its numerical approximation obtained via an adaptive composite Simpson’s
quadrature rule. 
These results extend the analysis  of Bornemann which focused
on  Fredholm determinants of trace class operators defined by 
scalar-valued kernels on a finite interval. 
Numerical results are provided for a Birman-Schwinger operator
that characterizes the stability of stationary solutions of nonlinear wave equations.
\end{abstract}

\maketitle

\section{Introduction}

Fredholm determinants
were originally  introduced by Fredholm~\cite{Fred1903} to characterize the 
solvability of integral equations of the second kind, 
$(\mathcal I + z \mathcal K) \boldsymbol \phi = \boldsymbol \psi$.
Since then they have found applications across a wide range 
of applied mathematics and mathematical physics (see 
Simon~\cite{Simon} and Bornemann~\cite{Bornemann} for detailed
literature reviews). 
Inspired by the von-Koch formula for the determinant of a matrix,
Hilbert~\cite{hilbert1904grundzuge} obtained a series expansion of the regular
Fredholm determinant, $\operatorname{det}_1(\mathcal I + z \mathcal K)$,
 of an integral operator, $\mathcal K$, with  a scalar-valued 
kernel on a finite interval. The $n$-th term in this  expansion  
involves an $n$-dimensional integral of a $n\times n$ determinant
of a matrix constructed from the kernel.  Approximating these
integrals by Riemann sums and applying the von-Koch formula,
Hilbert approximated the  Fredholm determinant of $\mathcal K$
by the determinant of a large matrix. Just over  a century later, 
Bornemann~\cite{Bornemann} obtained an estimate
on the rate of convergence of these matrix determinants
to the regular Fredholm determinant of a trace class operator on $L^2([a,b],\mathbb C)$. 
This rate of convergence depends on the order of the approximating quadrature rule 
and the degree of smoothness of the kernel.
Such error estimates are essential for practitioners who need to numerically compute 
Fredholm determinants. 
In \cite[Section~8]{Bornemann} Bornemann also studied the case of operators
on $L^2([a,b],\mathbb C^k)$ with matrix-valued kernels by using a trick of
Fredholm~\cite{Fred1903} to convert a system
of integral equations on a bounded interval, $I$, to an equivalent single integral equation
on a disjoint union of $k$ translates of $I$. 

Several recent papers on analytical solutions of nonlinear wave equations
and methods to compute spectral properties of  differential equations
apply or mention Bornemann's numerical computation of Fredholm determinants.
 Grudsky, Remling and Rybkin~\cite{Grudsky2015Inverse} 
 derive an analytical solution to the Cauchy problem
 for the KdV equation  in terms of the Fredholm determinant of 
 a Hankel operator obtained using the inverse scattering transform.
 They comment that Bornemann's method 
 could be applied to numerically evaluate their solution.
Cafasso, Claeys and Girotti~\cite{cafasso2021fredholm} 
use Fredholm determinants of certain integral
operators   to obtain solutions of the Painlev\'e~II hierarchy. 
These Fredholm determinants arise in 
random matrix and statistical physics models. They comment that 
Bornemann's method 
would be a more tractable approach for computing solutions of the
Painlev\'e II hierarchy than  numerically evaluating  Painlev\'e transcendents. 
Scheel~\cite{Scheel20223Nonlinear} mentions the difficulty of numerically evaluating 
Fredholm determinants in his paper on nonlinear
eigenvalue methods for determining  the stability of nonlinear waves. 
Zhao and Barnett~\cite{zhao2015robust} develop a numerical method 
for the \lq\lq drum\rq\rq problem, which is the Dirichlet problem 
 for the Laplacian on a two-dimensional
domain. They characterize the eigenfrequencies of a drum as the zeros
of the Fredholm determinant of a double-layer potential operator 
obtained via a boundary integral equation. Following Bornemann~\cite{Bornemann}, they 
approximate this Fredholm determinant using the  determinant of an 
associated Nystr\"om matrix. They analyze the error in this approximation
and demonstrate exponential convergence of the numerical root of the
matrix determinant  to the true eigenfrequency.

Our interest in  Fredholm determinants  stems from applications to the
stability of stationary and breather
solutions of nonlinear wave equations such as the 
complex Ginzburg-Landau equation~\cite{Kap,EssSpec}. 
It is well known~\cite{Kap} that the set of eigenvalues of
the linearization of the complex Ginzburg-Landau equation
about  a stationary (soliton) solution
is given by the set of zeros of the Evans function, which is a complex analytic function
of the spectral parameter.  
Gesztesy, Latushkin, and Makarov~\cite{EJF} 
established a general theory which showed that 
the  Evans function is closely related to the 2-modified
Fredholm determinant of a certain Birman–Schwinger operator, 
$\mathcal K=\mathcal K(\lambda)$, associated to the pulse. 
In particular, counted with multiplicity, 
the set of zeros of $\operatorname{det}_2(\mathcal I+\mathcal K(\lambda))$
coincides with the set of 
zeros of the Evans function, $E=E(\lambda)$. 
This Birman–Schwinger operator is defined in terms of a matrix-valued,
semi-separable Green’s kernel, $\textsf  K(x,y;\lambda)$, on the real line. 
In a forthcoming paper~\cite{Gallo}, we apply this theory to derive an explicit formula for the kernel
 $\textsf  K$ in the case of a stationary  solution of 
the complex Ginzburg–Landau equation and numerically study its 
Fredholm determinant.
The reason for our interest in 
Fredholm determinants of Birman–Schwinger operators is that they may
provide a means to determine the Floquet stability of time-periodic solutions of 
nonlinear wave equations, for which it is not possible to define or numerically calculate 
an Evans function~\cite{GeLaZu08}.

Motivated by this application, in this paper we  extend Bornemann's results 
to matrix-valued kernels on the entire real line. Our first main result is a  domain truncation
estimate relating the Fredholm determinant 
of an operator on $L^2(\mathbb R, \mathbb C^k) $ to that of the truncation 
of the operator to a finite interval. For this result we assume that the kernel
decays exponentially at infinity.
In related work using the Evans function, Sandstede and Scheel~\cite{SANDSTEDE2000233} studied the effects that truncating an unbounded domain to a finite interval has on the stability properties of nonlinear waves.
Our second main result is a quadrature error estimate
relating the Fredholm determinant of an integral operator with 
a matrix-valued kernel on a finite interval
to its numerical approximation  obtained via an adaptive composite Simpson's quadrature rule.
This result gives the rate of convergence of the numerical Fredholm determinant to the
true  determinant  as the maximum grid spacing in the quadrature rule approaches zero.
This result is somewhat different from the corresponding result of Bornemann
which concerns  the rate of convergence of the numerical Fredholm determinant to the
true  determinant as the order of the quadrature rule increases to infinity. 
Estimates like Bornemann's are useful 
in situations where a function is approximated by a single high-degree polynomial on the entire domain $[a,b]$.  However,  it is often more practical to 
employ a composite quadrature rule in which the domain $[a,b]$ is divided into 
subintervals and the function is separately approximated  by a low degree polynomial on each subinterval.

In section~\ref{sec:SchattenFred} we review the modern approach to the definition of the regular and 2-modified Fredholm determinants of trace class and Hilbert-Schmidt operators on a separable Hilbert space. With this approach, the regular Fredholm determinant of a
trace class operator,
$\mathcal K$, is defined 
in terms of the traces of  wedge products of $\mathcal K$.
In section~\ref{sec:FredDetIntegral} we 
derive analytical formulae for these  determinants in the case of 
operators with matrix-valued kernels on the entire real line. 
Bornemann~\cite{Bornemann} extended Fredholm's formula for the regular Fredholm determinant to matrix-valued kernels on a finite interval 
 by exploiting the equivalence between a $k\times k$
system of integral equations on a finite interval and a single integral equation on a 
disjoint union of $k$ copies of that interval.
Here, we take a different approach, which is to explicitly calculate the trace
of the operator ${\wedge}^n\mathcal{K}$. 
This approach enables us to establish the
formula for matrix-valued kernels on the entire real line, and also to handle
the extension to 2-modified Fredholm determinants.
 In section~\ref{vonKochsection}
we derive analogous formulae for  matrix-valued kernels defined on a finite discrete domain.
Such finite domains should be regarded as being discretizations of  finite intervals.
The formulae we obtain generalize the classical von-Koch formula for the determinant of a matrix.
In section~\ref{quadsection} we use these formulae  to define numerical approximations to  Fredholm determinants.
In section~\ref{truncsection}, we use the results of sections~\ref{sec:FredDetIntegral} and \ref{vonKochsection} to 
derive the domain truncation estimate. 
 Using a similar approach, in section~\ref{quadsection}, we derive the quadrature error
 estimate.
 Finally, to verify these error estimates, in section~\ref{sec:Results}, we show numerical results for a Birman-Schwinger operator
whose Fredholm determinant characterizes the linear stability of the hyperbolic secant
solution of the nonlinear Schr\"odinger equation. 
Throughout this paper, $\mathcal K$ denotes an operator, $\textsf K$ a (matrix-valued) kernel,
and $\mathbf K$ a matrix.

 \section{Schatten class operators and Fredholm determinants}\label{sec:SchattenFred}

In this section we review background material from 
Teschl~\cite{TeschlFA}, Simon~\cite{Simon} and 
Gohberg, Goldberg and Krupnik~\cite{GGK}
on the spaces of  trace class and Hilbert-Schmidt operators and  their Fredholm determinants. 

Let $\mathsf{H}$ be a real or complex separable 
Hilbert space with inner product $\langle \cdot, \, \cdot\rangle_\sH$, which is conjugate linear in the first slot and linear in the second slot.  
We let $\cB_\infty(\sH)$ denote the space of compact operators on $\sH$,
and let  $\cK\in \cB_\infty (\sH)$.
Since $\cK\cK^*$ and $\cK^*\cK$ are compact, self-adjoint operators, 
by the spectral theorem there exist orthonormal bases of eigenfunctions 
$\{\psi_\ell\}_{\ell=1}^\infty$ and $\{\phi_\ell\}_{\ell=1}^\infty$ such that 
\begin{align}
\label{KK1}
\cK^*\cK\phi_\ell=\mu_\ell^2\phi_\ell, \,\psi_\ell=\frac1{\mu_\ell}\cK\phi_\ell,\,
\cK\cK^*\psi_\ell=\mu_\ell^2\psi_\ell,\, \phi_\ell=\frac1{\mu_\ell}\cK^*\psi_\ell,
\end{align}
where $\mu_\ell$ are the singular values of  $\cK$, that is, the eigenvalues of the operator $(\cK\cK^*)^{1/2}$, ordered such that $\mu_1\ge\mu_2\ge\dots \geq 0$, $\mu_\ell\to0$ as $\ell\to\infty$.
Furthermore,
\begin{equation}\label{Kdecomp}
    \mathcal{K} = \sum_{j=1}^\infty \mu_j \langle \phi_j, \cdot \rangle \psi_j.
\end{equation}

The \emph{Schatten $p$-class} is the subspace of $\cB_\infty(\sH)$ defined by 
\begin{equation}
    \mathcal{B}_p(\sH) = \{ \mathcal{K} \in \cB_\infty(\sH) : \|\mathcal{K}\|_p < \infty\}, 
\end{equation}
where 
\begin{equation}
    \|\mathcal{K}\|_p = \left( \sum_{j=1}^{\infty} |\mu_j|^p \right)^{1/p} =: \|\mu\|_{\ell^p}.
\end{equation}
The two most important examples are the spaces $\cB_1(\sH)$  of \emph{trace class} operators and  $\cB_2(\sH)$ of \emph{Hilbert-Schmidt} operators. Since $\ell_1 \subset \ell_2$, all trace class
operators are Hilbert-Schmidt.

If $\mathcal{K}$ is trace class, then the \emph{trace} of $\cK$ 
is the linear transformation, 
$\Tr: \cB_1(\sH) \rightarrow \mathbb{C}$,  defined by
\begin{equation}\label{eq:traceprod}
\operatorname{Tr}(\mathcal K) \,\,=\,\, \sum_k
\langle h_k, \mathcal K h_k \rangle_\sH,
\end{equation}
where here and below $\{h_k\}$ is any orthonormal basis for $\mathcal H$. 
We recall that $\operatorname{Tr}(\mathcal{K})$ is finite,
\begin{equation}
    |\operatorname{Tr}(\mathcal{K})| \leq \|\mathcal{K}\|_1,
\end{equation}
and  is independent of the choice of orthonormal basis.
There are two other important formulae for the trace. The first states
that 
\begin{equation}\label{eq:TraceByEval}
\operatorname{Tr}(\mathcal K) 
 \,\,=\,\, \sum\limits_\ell \mu_\ell \langle \phi_\ell, \psi_\ell\rangle_\sH,
 \end{equation}
where the $\{\mu_\ell\}_\ell$ are the singular values of $\cK$ and 
 $\{\phi_\ell\}_\ell$, and $\{\psi_\ell\}_\ell$ are the 
 orthonormal bases in \eqref{KK1}. 
The second result is Lidskii's Theorem which states that
\begin{equation}\label{eq:lidskii}
\operatorname{Tr}(\mathcal K) 
 \,\,=\,\,   \sum\limits_\ell \lambda_\ell,
\end{equation}
where $\{\lambda_\ell\}_\ell$ is the set of eigenvalues of $\cK$ counted
with algebraic multiplicity.

The Fredholm determinant of a trace class operator, $\mathcal{K}$, is defined in terms of the traces of the induced operators, ${\wedge}^n\mathcal{K},$ on the wedge product spaces, ${\wedge}^n \sH$. Summarizing~\cite[Section 1.5]{Simon},
 let $\bigotimes^n \sH = \sH \otimes \dots \otimes \sH$ be the $n$-fold tensor product of $\sH.$
Given any $\psi_1, \dots, \psi_n \in \sH,$ we define
\begin{equation}\label{psiwedge}
    \psi_1 \wedge \dots \wedge \psi_n = \frac{1}{\sqrt{n!}} \sum_{\pi \in \sigma_n} (-1)^{\pi} \psi_{\pi(1)} \otimes \dots \otimes \psi_{\pi (n)},
\end{equation}
where $\sigma_n$ is the set of permutations on $\{1,\dots,n\}$ and $(-1)^{\pi}$ is the sign of the permutation $\pi.$ For $n>0$, 
$\wedge^n \sH$ is defined to be
 the closure in $\bigotimes^n \sH$ of the set of  all finite linear combinations of the form $\psi_1 \wedge \dots \wedge \psi_n$,  and $\wedge^0 \sH = \mathbb{C}$. 
Because of the identity
\begin{equation}\label{eq:detfromwedge}
    (\phi_1 \wedge \dots \wedge \phi_n, \psi_1 \wedge \dots \wedge \psi_n) = \det((\phi_j,\psi_i)_{1 \leq i,j \leq n}),
\end{equation}
if $\{h_k\}_{k=1}^{\infty}$ is an orthonormal basis of $\sH$, then 
$\{h_{j_1} \wedge \dots \wedge h_{j_n}\}_{j_1 < \dots < j_n}$ is an orthonormal basis for $\wedge^n \sH$.

If $\mathcal{K} : \sH \rightarrow \sH$ is a linear operator, then there exists an induced linear operator $\mathcal{K} \otimes \dots \otimes \mathcal{K} : \bigotimes^n \sH \rightarrow \bigotimes^n \sH$ so that 
\begin{equation}\label{Kcrossphi}
    (\mathcal{K} \otimes \dots \otimes \mathcal{K})(\phi_1 \otimes \dots \otimes \phi_n) = \mathcal{K}\phi_1 \otimes \dots \otimes \mathcal{K}\phi_n.
\end{equation}
Clearly, $\mathcal{K} \otimes \dots \otimes \mathcal{K}$ maps $\wedge^n \sH$ into $\wedge^n \sH.$ We denote this operator by 
\begin{equation}
 {\wedge}^n\mathcal{K} : {\wedge}^n \sH \rightarrow {\wedge}^n \sH.
\end{equation}

We observe that if $\mathcal{K} : \sH \rightarrow \sH$ is compact with eigenvalues $\{\lambda_j\}_{j=1}^{\infty},$ then the set of eigenvalues of $\wedge^n\mathcal{K}$ is given by $\prod_{j_1 < \dots < j_n} \lambda_{j_k}$. Consequently, by  Lidskii's theorem\eqref{eq:lidskii}, 
\begin{equation}
    \Tr({\wedge}^k\mathcal{K}) = \sum_{j_1 < \dots < j_k} \lambda_{j_1} \dots \lambda_{j_k}.
\end{equation}
In the special case that $\dim \sH = N < \infty$,  we have that
\begin{equation}\label{tracedet}
\det(\mathcal{I} + \mathcal{K}) \,\,=\,\,  \prod_{\ell=1}^N (1 + \lambda_\ell) 
\,\,=\,\,   \sum_{n=0}^N \Tr({\wedge}^n\mathcal{K}).
\end{equation}
This identity motivates the definition of the regular Fredholm determinant
of  a trace class operator.

\begin{prop}\cite[Lemma~3.3, Theorems 3.5 and 3.7]{Simon}
Suppose $\mathcal{K}$ is a trace class operator on separable Hilbert space $\sH.$ Then $\wedge^n\mathcal{K}$ is a trace class operator on $\wedge^n \sH$ and
 \begin{equation}
     \left\|{\wedge}^n\mathcal{K}\right\|_1 \leq \frac{1}{n!}\|\mathcal{K}\|_1^n.
 \end{equation}
 Consequently, the regular Fredholm determinant of $\mathcal{K},$ which is 
 defined by the series
 \begin{equation}
     {\det}_1(\mathcal{I} + z\mathcal{K}) := \sum_{n=0}^{\infty}z^n \Tr({\wedge}^n\mathcal{K}),
 \end{equation}
 converges uniformly and absolutely  to an entire function of $z$ such that
 \begin{equation}
     |{\det}_1(\mathcal{I} +z\mathcal{K})| \leq \exp(|z|\|\mathcal{K}\|_1).
 \end{equation}
Furthermore,
 $\mathcal{I} + \mathcal{K}$ is invertible if and only if ${\det}_1(\mathcal{I} + \mathcal{K}) \neq 0$, and 
 \begin{equation}
 {\det}_1(\mathcal{I} + \mathcal{K}) = \prod_{\ell=1}^\infty (1 + \lambda_\ell).
 \end{equation}
\end{prop}

If $\cK$ is Hilbert-Schmidt but not trace class it is still possible to define
a Fredholm determinant  which characterizes the invertibility of $\cI+\cK$.

\begin{prop}\cite[Lemma 9.1, Theorem 9.2]{Simon}\label{prop:det2}
Let $\mathcal{K} \in \cB_2(\sH)$ and let $\exp(\cK)$ be the exponential
of $\cK$, defined by the standard power series. 
Then the operator 
$ R_2(\mathcal{K}) = (\cI + \mathcal{K})\exp(-\mathcal{K}) - \cI$ is 
trace class and the 2-modified Fredholm determinant of $\cK$ is 
defined by
\begin{equation}
    {\det}_2(\mathcal{I} + z\mathcal{K}) := {\det}_1[\cI + R_2(z\mathcal{K})] =  {\det}_1[(\cI+z\mathcal{K})\exp(-z\mathcal{K})].
\end{equation}
Furthermore, 
\begin{enumerate}
    \item[(1)] ${\det}_2(\mathcal{I} + \mathcal{K}) = \prod_{\ell}\left[ (1 + \lambda_\ell) e^{-\lambda_\ell} \right]$,
    \item[(2)] $|{\det}_2(\mathcal{I} + \mathcal{K})| \leq \emph{exp}(\Gamma_2 \|\mathcal{K}\|_2^2),$ for some constant  $\Gamma_2$,
    \item[(3)] If $\cK_1,\cK_2\in \cB_2(\sH)$, then 
     \begin{equation}
     |{\det}_2(\mathcal{I} + \mathcal{K}_1) - {\det}_2(\mathcal{I} + \mathcal{K}_2)| \leq \|\mathcal{K}_1 - \mathcal{K}_2\|_2 \emph{exp}(\Gamma(\|\mathcal{K}_1\|_2 + \|\mathcal{K}_2\|_2 + 1)^2),
     \end{equation}
    \item[(4)] If $\mathcal{K} \in \cB_1,$ then \begin{equation}\label{det2exp}
{\det}_2(\mathcal{I} + \mathcal{K}) = e^{-\operatorname{Tr}(\mathcal{K})}
{\det}_1(\mathcal{I} + \mathcal{K}),
            \end{equation} 
     and 
    \item[(5)] $\mathcal{I} + \mathcal{K}$ is invertible if and only if ${\det}_2(\mathcal{I} + \mathcal{K}) \neq 0$.
\end{enumerate}
\end{prop}

\section{Fredholm Determinants of  Integral Operators}\label{sec:FredDetIntegral}

In Section~\ref{sec:SchattenFred}, 
 we defined  the trace and regular Fredholm determinant of a trace class operator 
 and the 2-modified Fredholm determinant of a Hilbert-Schmidt operator
 on an abstract Hilbert space.  
 In this section, we  consider Hilbert-Schmidt integral operators, $\cK$,
 on  $L^2(\mathbb R,\mathbb{C}^k)$ that are defined in terms of
  matrix-valued kernels, $\textsf{K}=\textsf{K}(x,y)$. 
  Building on a classical result of Weidmann~\cite{weidmann1966integraloperatoren},
  in~\cite{zweck2024regularity} 
  we showed that such a Hilbert-Schmidt operator, $\cK$, is trace class provided that
   the associated kernel, $\textsf K$, is Lipschitz continuous and that
 $\textsf K$ and  its partial derivatives 
  decay  exponentially  away from the diagonal. 
 The goal of this section is to derive integral formulae for the
 regular and 2-modified Fredholm determinants of the operator 
 $\cK$ in terms of the kernel $\textsf K$.  
These formulae generalize a well-known formula of Fredholm~\cite{Fred1903,Simon}
 for the regular Fredholm determinant of a trace class operator on
$L^2(\mathbb [a,b],\mathbb{C})$.

Let $X=[a,b]$ be a finite interval or $X=\mathbb R$.
We let $L^2(X,\mathbb{C}^k)$ denote the space of square-integrable functions from $X$ to $\mathbb{C}^k$ with the inner product
\begin{equation}
    \langle \boldsymbol{\phi}, \boldsymbol{\psi} \rangle = \int_{\mathbb{R}} \boldsymbol{\phi}^*(x) \boldsymbol{\psi}(x) dx,
\end{equation}
where $\boldsymbol{\psi}$ is $k \times 1$ and $\boldsymbol{\phi}^* = \overline{\boldsymbol{\phi}}^{\,T}$ is $1\times k$.

The following result~\cite{BirmanSolomjak1987} formalizes the relationship between an operator and its kernel.

\begin{prop}
Let $\emph{\textsf K}\in L^2(X\times X, \mathbb C^{k\times k})$ 
be a square-integrable matrix-valued kernel on $X$ and let
$\boldsymbol\psi\in L^2(X,\mathbb C^k)$.
Then for almost all $x\in X$, we have 
$\emph{\textsf K}(x,\cdot)\boldsymbol\psi(\cdot) \in L^1(X, \mathbb C^k)$
and  so
\begin{equation}\label{cKdef2}
(\cK \boldsymbol\psi)(x)=\int_X \emph{\textsf K}(x,y)\boldsymbol\psi(y)\,dy
\end{equation}
is defined. Moreover, $\cK \in \cB_2(L^2(X,\mathbb C^k ))$ 
is a Hilbert-Schmidt operator
and 
$\| \cK \|_{\cB_2(L^2(X, \mathbb C^k))} = \| \emph{\textsf K} \|_{L^2(X\times X,  \mathbb C^{k\times k}) }$.
Conversely, for every $\cK \in \cB_2(L^2(X, \mathbb C^k))$ there is a unique 
$\emph{\textsf K}\in L^2(X\times X, \mathbb C^{k\times k})$ for which \eqref{cKdef2} holds.
\end{prop}

\begin{rem}
Since every trace class operator is Hilbert-Schmidt, every trace class operator
on $L^2(X, \mathbb C^k)$ has a matrix-valued kernel. 
\end{rem}

The integral formula we seek for the trace of  $\wedge^n\cK$ depends on
the following basic result, whose proof can be found for example in Simon~\cite{Simon}.
 
\begin{prop}\label{prop:TraceIntegralFormulaScalarFinite}
Suppose that $\mathcal{K} \in \cB_1(L^2([a,b],\mathbb{C}))$ is of the form
\begin{equation}
    (\mathcal{K}\phi)(x) = \int_a^b \emph{\textsf K}(x,y) \phi(y) dy,
\end{equation}
where the kernel $\emph{\textsf K}\in C^0([a,b]\times[a,b], \mathbb C)$ is continuous. Then \begin{equation}
    \operatorname{Tr}(\mathcal{K}) = \int_a^b \emph{\textsf K}(x,x)dx.
\end{equation}
\end{prop}
 
First, we extend Proposition~\ref{prop:TraceIntegralFormulaScalarFinite} to the case of an integral operator 
$\mathcal{K}$  on all of $\mathbb{R}$ instead of on a finite interval  $[a,b]$.

\begin{prop}\label{prop:TraceIntegralFormulaScalarReal}
Let $\mathcal{K} \in \cB_1(L^2(\mathbb{R},\mathbb{C}))$ be of the form
\begin{equation}
    (\mathcal{K}\phi)(x) = \int_{\mathbb{R}} \emph{\textsf K}(x,y) \phi(y) dy,
\end{equation}
where $\emph{\textsf K} = \emph{\textsf K}(x,y) \in L^2(\mathbb{R} \times \mathbb{R}, \mathbb{C}) \cap C^{\,0}(\mathbb{R} \times \mathbb{R},\mathbb{C}),$ and suppose that $F(x) := \emph{\textsf K}(x,x) \in L^1(\mathbb{R},\mathbb{C}).$ Then
\begin{equation}
    \operatorname{Tr}(\mathcal{K}) = \int_{\mathbb{R}} \emph{\textsf K}(x,x) dx.
\end{equation}
\end{prop}

\begin{proof}
Let $\mathcal{Q}_n: L^2(\mathbb{R},\mathbb{C}) \rightarrow L^2(\mathbb{R},\mathbb{C})$ be defined by 
\begin{equation}
    (\mathcal{Q}_n \phi)(x) = \bigchi_{[-n,n]}(x)\phi(x),
\end{equation}
where $\bigchi_{[-n,n]}$ is the characteristic function of $[-n,n] \subset \mathbb{R}.$ We can construct an orthonormal basis $\{\psi_m\}_{m=1}^{\infty}$ of $L^2(\mathbb{R},\mathbb{C})$
so that for every  $n,$ there is a subsequence $\{\psi_{m_k^{(n)}}\}_{k=1}^{\infty},$ for which
\begin{equation}
    \text{Span}\{ \psi_{m_ k^{(n)}} : k=1,2,\dots\} = \text{Ran}(\mathcal{Q}_n) = L^2([-n,n],\mathbb{C}),
\end{equation}
and
\begin{equation}
    m^{(n)} := \left\{m_k^{(n)}\right\}_{k=1}^{\infty} \subseteq \left\{m_k^{(n+1)}\right\}_{k=1}^{\infty} \subseteq \dots \subseteq \mathbb{N},
\end{equation}
with $\bigcup_{n=1}^{\infty} m^{(n)} = \mathbb{N}$. Now,
\begin{equation}
    \operatorname{Tr}(\mathcal{Q}_n\mathcal{K}\mathcal{Q}_n) = \sum_{\ell \in m^{(n)}} \langle \psi_{\ell}, \mathcal{K} \psi_{\ell} \rangle,
\end{equation}
as $\{\psi_{\ell} : \ell \in m^{(n)}\}$ is an orthonormal basis for 
$L^2([-n,n],\mathbb C)$.  By Proposition~\ref{prop:TraceIntegralFormulaScalarFinite},   we have that
\begin{equation}\label{eq:TrLimit}
\operatorname{Tr}(\mathcal{Q}_n \mathcal{K} \mathcal{Q}_n)
= \int_{\mathbb{R}} \bigchi_{[-n,n]}(x) {\textsf K}(x,x) dx. 
\end{equation}
Let $\epsilon > 0$. Since $\mathcal{K}$ is trace class, $\exists M \in \mathbb{N}$ such that 
  $  \sum_{\ell > M} |\langle \psi_{\ell}, \mathcal{K}\psi_{\ell} \rangle| < \epsilon$.
Since $\bigcup_{n=1}^{\infty} m^{(n)} = \mathbb{N}, \exists N$ such that $\{1,\dots,M\} \subseteq m^{(N)}.$ 
So, by \eqref{eq:traceprod},
\begin{equation}
    |\operatorname{Tr}(\mathcal{K}) - \operatorname{Tr}(\mathcal{Q}_N \mathcal{K} \mathcal{Q}_N)| \leq \sum_{\ell \notin m^{(N)}} |\langle \psi_{\ell}, \mathcal{K}\psi_{\ell} \rangle|
    \leq \sum_{\ell > M}|\langle \psi_{\ell}, \mathcal{K} \psi_{\ell} \rangle| < \epsilon.
\end{equation}
Therefore, by  \eqref{eq:TrLimit} and the  Lebesgue Dominated Convergence Theorem,
\begin{equation}
\operatorname{Tr}(\mathcal{K}) 
= \lim_{n \rightarrow \infty} \int_{\mathbb{R}} \bigchi_{[-n,n]}(x) {\textsf K}(x,x) dx
=\label{TrKintK} \int_{\mathbb{R}} {\textsf K}(x,x) dx.
\end{equation}
\end{proof}

Next, we further generalize  Propositions \ref{prop:TraceIntegralFormulaScalarFinite} and \ref{prop:TraceIntegralFormulaScalarReal} to the case of
an operator $\mathcal{K}$ with a matrix-valued kernel.

\begin{thm}\label{thm:MatrixKernelTrace}
Let $X = [a,b]$ or $X = \mathbb{R}.$ Suppose that
 $\mathcal{K} \in \cB_1(L^2(X,\mathbb{C}^k))$ is of the form
\begin{equation}
    (\mathcal{K}\boldsymbol{\phi})(x) = \int_{X}\emph{\textsf{K}}(x,y) \boldsymbol{\phi}(y) dy, 
    \qquad \boldsymbol{\phi} \in L^2(X,\mathbb C^k),
\end{equation}
where the kernel $\emph{\textsf{K}} \in  L^2(X \times X, \mathbb{C}^{k \times k}) \cap 
C^0(X \times X, \mathbb{C}^{k \times k})$
is continuous, and suppose that
$F(x) := \operatorname{Tr}(\emph{\textsf{K}})(x,x) \in L^1(X,\mathbb{C}).$ Then
\begin{equation}\label{trboldK}
    \operatorname{Tr}(\mathcal{K}) = \int_X \operatorname{Tr}(\emph{\textsf{K}})(x,x) dx.
\end{equation}
\end{thm}

\begin{proof}
We define
\begin{equation}\label{eq:defKtr}
    \left( \mathcal{K}^{\operatorname{Tr}} \phi \right)(x) := \int_X \operatorname{Tr}(\textsf{K})(x,y) \phi(y) dy,  \qquad {\phi} \in L^2(X,\mathbb C).
\end{equation}
We observe that $\mathcal{K}^{\operatorname{Tr}} \in \cB_2(L^2(X,\mathbb{C}))$,
since  $\operatorname{Tr}(\textsf{K}) \in L^2(X \times X, \mathbb{C})$
as ${\textsf{K}} \in  L^2(X \times X, \mathbb{C}^{k \times k})$.
By \eqref{KK1} and \eqref{eq:defKtr},
\begin{equation}
    \mathcal{K}^{\operatorname{Tr}} \phi (x) 
    = \int_X \sum_{\ell=1}^\infty \mu_\ell \boldsymbol{\phi}_\ell^*(y) \boldsymbol{\psi}_\ell (x) \phi(y) dy
   = \sum_{\ell=1}^\infty \sum_{j = 1}^k  \mu_\ell \langle \phi_{\ell j}, \phi \rangle \psi_{\ell j}(x), 
\end{equation}
where $\phi_{\ell j}$ is the $j$-th component of  $\boldsymbol{\phi}_\ell$.
Therefore, $\mathcal{K}^{\operatorname{Tr}} \in \cB_1(L^2(X,\mathbb C))$, since it
admits a representation as in  \eqref{Kdecomp}
 with an absolutely convergent series of singular values, $\mu_\ell$, each with 
 multiplicity $k$.

Let $\{\phi_m\}_{m=1}^\infty$ be an orthonormal basis for 
$L^2(X,\mathbb C)$ and let $\{ \mathbf e_j\}_{j=1}^k$
be the standard basis for $\mathbb C^k$. Since
$\boldsymbol{\phi}_{m,j}(x)  := \phi_m(x) \mathbf e_j$ is an orthonormal
basis for $L^2(X,\mathbb C^k)$, by \eqref{eq:traceprod}
\begin{align}
\operatorname{Tr}(\mathcal{K}) 
&\,\,=\,\, \sum_{m=1}^{\infty} \sum_{j=1}^k \langle \boldsymbol{\phi}_{m,j}, \mathcal{K} \boldsymbol{\phi}_{m,j} \rangle_{L^2(\mathbb {R},\mathbb{C}^k)}
\nonumber
\\
&\,\,=\,\, \sum_{m=1}^{\infty} \int_X \int_X \phi_m^*(x) \left[ \sum_{j=1}^k \mathbf{e}_j^T \textsf{K}(x,y) \mathbf{e}_j \right] \phi_m(y) dy \,dx
\nonumber
\\
&\,\,=\,\, \sum_{m=1}^{\infty} \int_X \int_X \phi^*_m(x) \operatorname{Tr}(\textsf{K})(x,y) \phi_m(y) dy\, dx
\nonumber
\\
&\,\,=\,\, \sum_{m=1}^{\infty} \langle \phi_m, \mathcal{K}^{\operatorname{Tr}} \phi_m \rangle_{L^2(X,\mathbb{C})}
\,\,=\,\, \operatorname{Tr}(\mathcal{K}^{\operatorname{Tr}}).
\end{align}
Since the assumptions of Propositions~\ref{prop:TraceIntegralFormulaScalarFinite} and \ref{prop:TraceIntegralFormulaScalarReal} hold for the
scalar-valued kernel $\operatorname{Tr}(\textsf{K})$, the result now follows.
\end{proof}

Finally, we derive a formula for the trace of the $n$-th wedge product of a trace class integral  operator $\mathcal{K}$ with a matrix-valued kernel
and hence for the 
regular Fredholm determinant  $\det_1(\mathcal{I} + \mathcal{K})$. 
For simplicity, 
in the case $X=\mathbb R$ we make the following hypothesis.

\begin{hypothesis}\label{hyp9}
The operator $\mathcal{K} \in \cB_2(L^2(\mathbb{R},\mathbb{C}^k))$ is of the form
\begin{equation}
    (\mathcal{K} \boldsymbol{\phi})(x) = \int_{\mathbb{R}} \textsf{K}(x,y) \boldsymbol{\phi}(y) dy,
    \quad\text{for }\boldsymbol{\phi} \in L^2(\mathbb{R},\mathbb{C}^k),
\end{equation}
where  $\textsf{K} \in L^2(\mathbb{R} \times \mathbb{R}, \mathbb{C}^{k \times k}) \cap C^0(\mathbb{R} \times \mathbb{R}, \mathbb{C}^{k \times k})$ has the property that 
 \begin{equation}\label{eq:HypExpDecay}
    \|\textsf{K}(x,y)\|_{\mathbb{C}^{k \times k}} \leq Ce^{-a |x|}e^{-b|y|}, \,\, \forall x, y \in \mathbb{R},
\end{equation}
for some $a,b,C > 0$.
\end{hypothesis}

\begin{thm}\label{prop:TraceIntegralFormulaScalarFinite0}
Let $\mathcal{K} \in \cB_1(L^2(\mathbb{R}, \mathbb{C}^k))$ satisfy Hypothesis \ref{hyp9}. For each $n \in \mathbb{N},$ let 
\begin{equation}\label{Kdetnbyn}
    {K}_{\mathbf{j}}^{(n)}(\mathbf{x}, \mathbf{y}) := \det \left( \left[ \emph{\textsf{K}}_{j_{\alpha}j_{\beta}}(x_{\alpha}, y_{\beta}) \right]_{\alpha, \beta = 1}^n \right),
\end{equation}
where $\mathbf{x} = (x_1, \dots, x_n), \, \mathbf{y} = (y_1, \dots, y_n),$ and 
\begin{equation}\nonumber
    \mathbf{j} \in J_k^{(n)} = \{ (j_1 \dots j_n) : 1 \leq j_{\alpha} \leq k, \,\, \forall \alpha \},
\end{equation}
is an index set of cardinality  $|J_k^{(n)}| = k^n.$ Then
\begin{equation}\label{prop:TraceIntegralFormulaScalarFinite0tr}
    \operatorname{Tr}({\wedge}^n \mathcal{K}) = \sum_{\mathbf{j} \in J_k^{(n)}} \frac{1}{n!} \int_{\mathbb{R}^n} {K}_{\mathbf{j}}^{(n)}(\mathbf{x}, \mathbf{x}) d\mathbf{x}.
\end{equation}
Consequently, 
\begin{equation}\label{prop:TraceIntegralFormulaScalarFinite0det}
    {\det}_1(\mathcal{I} + z\mathcal{K}) = \sum_{n=0}^{\infty} \frac{z^n}{n!} \sum_{\mathbf{j} \in J_k^{(n)}} \int_{\mathbb{R}^n} {K}_{\mathbf{j}}^{(n)}(\mathbf{x}, \mathbf{x}) d\mathbf{x}.
\end{equation}
\end{thm}

\begin{rem}
The same result holds in the case of a finite interval, $X=[a,b]$, without the
need for \eqref{eq:HypExpDecay} in Hypothesis~\ref{hyp9}.
\end{rem}

\begin{rem}
    To unpack the definition of $K_{\mathbf{j}}^{(n)}(\mathbf{x},\mathbf{y})$ in (\ref{Kdetnbyn}), we observe that $[\textsf{K}_{j_{\alpha} j_{\beta}}(x_{\alpha},y_{\beta})]_{\alpha, \beta = 1}^n $ is the $n \times n$ matrix whose $(\alpha, \beta)$-entry, $\textsf{K}_{j_{\alpha} j_{\beta}}(x_{\alpha},y_{\beta}),$ is the $(j_{\alpha},j_{\beta})$-entry of the $k \times k$ matrix-valued kernel $\textsf K$ evaluated at the point $(x_{\alpha},y_{\beta}) \in \mathbb{R}^2.$ Here, $x_{\alpha}, y_{\beta}$ are the $\alpha$ and $\beta$ entries of the vectors $\mathbf{x}, \mathbf{y} \in \mathbb{R}^n.$ Therefore the formula for ${\det}_1(\mathcal{I} + z\mathcal{K})$ in (\ref{prop:TraceIntegralFormulaScalarFinite0det}) is  an infinite sum over an index $n \in \mathbb{N}$ and a multi-index $\mathbf{j} \in J^{(n)}_k$ of integrals over $\mathbb{R}^n$ of determinants of $n \times n$ matrices.  Needless to say, this formula cannot be used in numerical calculations! However, in Sections~\ref{truncsection} and \ref{quadsection} we will use it to  establish the
 convergence properties of a numerical approximation to 
 ${\det}_1(\mathcal{I} + z\mathcal{K})$. 
\end{rem}

\begin{proof}
First, we establish (\ref{prop:TraceIntegralFormulaScalarFinite0tr}) for $\operatorname{Tr}(\wedge^n \mathcal{K})$ over the finite interval $I=[0,1]$.
For each $N,$ let  $\mathcal B_N = \{ \boldsymbol\phi_{m,j}  = \phi_m \mathbf e_j: 1 \leq m \leq 2^N, \, 1 \leq j \leq k\}$
 be the orthonormal set in $L^2(I, \mathbb{C}^k)$ for which
\begin{equation}
    \phi_m(x) = 
    \begin{cases}
     2^{N/2} \quad & \text{if } \frac{m-1}{2^N} \leq x < \frac{m}{2^N},\\
    0 \quad &\text{otherwise},
    \end{cases}
\end{equation}
and let
\begin{equation}
    \boldsymbol{\Phi}_{\mathbf{m},\mathbf{j}} = \boldsymbol{\phi}_{m_1, j_1} \wedge \dots \wedge \boldsymbol{\phi}_{m_n, j_n}\,\,\in\,\,\wedge^n(L^2(I,\mathbb C^k)).
\end{equation}
Let
\begin{eqnarray}
\mathcal{M}_N^{(n)} &=& \{\mathbf{m} = (m_1 \dots m_n) : 1 \leq m_{\alpha} \leq 2^N, \,\, \forall \alpha\},\\
J_k^{(n)} &=& \{\mathbf{j} = (j_1 \dots j_n) : 1 \leq j_{\alpha} \leq k, \,\, \forall \alpha\}.
\end{eqnarray}
Employing the lexicographic ordering that $(m_1, j_1) < (m_2, j_2) $
if  $m_1 < m_2$ or  $m_1=m_2$ and $ j_1 < j_2$,
we define
\begin{equation}
    \mathcal{M}_N^{(n), \uparrow} = \{ (\mathbf{m}, \mathbf{j}) \in \mathcal{M}_N^{(n)} \times  J_k^{(n)}
  : (m_1, j_1) < (m_2, j_2) < \dots < (m_n, j_n) \}. 
\end{equation}
Then $\wedge^n\mathcal{B}_N := \{\boldsymbol{\Phi}_{\mathbf{m}, \mathbf{j}} : (\mathbf{m}, \mathbf{j}) \in \mathcal{M}_N^{(n), \uparrow}\}$ is an orthonormal set in $\wedge^n(L^2(I,\mathbb{C}^k))$. Let $\mathcal{P}_N$ be the projection onto $\text{Span}( \wedge^n\mathcal{B}_N)$.
Then, by Lemma \ref{lemmatracesum} below,
\begin{align}\label{37*}
    \operatorname{Tr}({\wedge}^n \mathcal{K}) &\,\,=\,\, 
    \lim_{N \rightarrow \infty} \operatorname{Tr}(\mathcal{P}_N\left({\wedge}^n \mathcal{K}\right)\mathcal{P}_N)  \nonumber \\
    &\,\,=\,\, \frac{1}{n!} \sum_{\mathbf{m} \in \mathcal{M}_N^{(n)}} \sum_{\mathbf{j} \in J_k^{(n)}} \langle \boldsymbol{\Phi}_{\mathbf{m}, \mathbf{j}}, \left( {\wedge}^n \mathcal{K} \right) \boldsymbol{\Phi}_{\mathbf{m}, \mathbf{j}} \rangle_{{\wedge}^n (L^2(I,\mathbb{C}^k))},
\end{align}
is a sum over the much larger index set  $\mathcal{M}_N^{(n)} \times  J_k^{(n)}$.

Let $\chi_{\alpha, \beta} = \operatorname{support}(\phi_{m_\alpha} \otimes
\phi_{m_\beta})$.  Then,
\begin{align} \nonumber 
  \langle \phi_{m_{\alpha}} \mathbf{e}_{j_{\alpha}}, \mathcal{K}\left( \phi_{m_{\beta}} \mathbf{e}_{j_{\beta}}\right) \rangle_{L^2(I,  \mathbb{C}^k)} 
    &= \langle \phi_{m_{\alpha}} \mathbf{e}_{j_{\alpha}}, \int_I \sum_{i=1}^k \textsf{K}_{ij_{\beta}}(\cdot,y) \phi_{m_{\beta}}(y) dy \,\,\mathbf{e}_{i} \rangle\\ \nonumber
    &= 2^N  \iint\limits_{\chi_{\alpha,\beta}} \sum_{i=1}^k (\mathbf{e}_{j_{\alpha}})^T \textsf{K}_{i j_{\beta}}(x,y)  \,\mathbf{e}_i\, dx dy\\  
    &    \cong \frac{1}{2^N} \textsf{K}_{j_{\alpha} j_{\beta}} \left(\frac{m_{\alpha}}{2^N}, \frac{m_{\beta}}{2^N}\right),
    \label{eq:littlesquareserror}
\end{align}
where the approximation is valid for large $N$ 
since the area of the support $\chi_{\alpha, \beta}$ of  $\phi_{m_\alpha} \otimes
\phi_{m_\beta}$ is 
 $2^{-2N}$, which is small for large $N$.
In particular, since $\textsf K$ is uniformly 
continuous on the compact set $I\times I$, the 
error in \eqref{eq:littlesquareserror} converges to zero as $N \to \infty$.

Next, by \eqref{37*}, \eqref{eq:littlesquareserror}, and \eqref{eq:detfromwedge}
we have that 
\begin{align}
\operatorname{Tr}\left( \mathcal{P}_N \left({\wedge}^n \mathcal{K} \right)\mathcal{P}_N \right) &\cong \frac{1}{n!} \sum_{\mathbf{m} \in \mathcal{M}_N^{(n)}} \sum_{\mathbf{j} \in J_k^{(n)}} \det\left[ \frac{1}{2^N} \textsf{K}_{j_{\alpha}j_{\beta}} \left( \frac{m_{\alpha}}{2^N}, \frac{m_{\beta}}{2^N}\right)\right]_{\alpha, \beta = 1}^n
\nonumber
\\
&=
\label{Riemannsum} \frac{1}{n!} \sum_{\mathbf{m} \in \mathcal{M}_N^{(n)}} \sum_{\mathbf{j} \in J_k^{(n)}} \left(\frac{1}{2^N}\right)^n K_{\mathbf{j}}^{(n)}\left( \frac{m_{\alpha}}{2^N}, \frac{m_{\beta}}{2^N}\right)
\\ \label{convsum}
&\cong \frac{1}{n!} \sum_{\mathbf{j}\in J_k^{(n)}} \int_{I^n} K_{\mathbf j}^{(n)}(\mathbf{x}, \mathbf{x}) d\mathbf{x},
\end{align}
where the error in (\ref{convsum}) converges to $0$ as $N\to\infty$ since (\ref{Riemannsum}) is a Riemann sum for the continuous function $F^{(n)}(x) := {K}_{\mathbf j}^{(n)}(\mathbf{x}, \mathbf{x}).$ This establishes the result in the case that $\mathcal{K} \in \cB_1(L^2(I,\mathbb{C}^k)),$ where $I$ is a finite interval. 

To prove (\ref{prop:TraceIntegralFormulaScalarFinite0det}) for operators on $\mathbb{R}$, we note that by  Hypothesis \ref{hyp9}, $F^{(n)}(\mathbf{x}) := K_{\mathbf j}^{(n)}(\mathbf{x}, \mathbf{x}) \in L^1(\mathbb{R}^n, \mathbb{C})$. 
Finally, arguing as in the proof of 
Proposition~\ref{prop:TraceIntegralFormulaScalarReal}, we obtain \eqref{prop:TraceIntegralFormulaScalarFinite0tr} and \eqref{prop:TraceIntegralFormulaScalarFinite0det}. 
\end{proof}

In the proof of Theorem \ref{prop:TraceIntegralFormulaScalarFinite0} above, we made use of the following lemma.

\begin{lem}\label{lemmatracesum}
Let $\mathcal{K}$ be a linear operator on a separable Hilbert space $\sH.$ For each  $N \in \mathbb{N}$, let $\mathcal{B}_N = \{\phi_{N,m}\}_{m=1}^{m_N}$ 
be an orthonormal set in $\sH$ chosen so that $\cup_{N=1}^{\infty} \mathcal{B}_N$ is an orthonormal basis for $\sH.$
Let 
\begin{equation}
    \Phi_{N,\mathbf{m}} = \phi_{N,m_1} \wedge \dots \wedge \phi_{N,m_n} \,\, \in \,\, \wedge^n\emph{\textsf H},
\end{equation}
where $\mathbf{m} = (m_1, \dots, m_n),$ and let
\begin{equation}
\mathcal{M}_N^{(n),\uparrow} = \{\mathbf{m} : 1 \leq m_1 < \dots < m_n \leq N\}.
\end{equation}
Then \begin{equation}
    {\wedge}^n \mathcal{B}_N := \{\Phi_{N,\mathbf{m}}\}_{\mathbf{m} \in \mathcal{M}_N^{(n), \uparrow}}
\end{equation}
is an orthonormal set in ${\wedge}^n \sH. $ 
Furthermore, if $\mathcal{P}_N$ is the projection onto \emph{Span}$(\wedge^n \mathcal{B}_N)$, then 
\begin{equation}\label{trPN}
    \operatorname{Tr}(\mathcal{P}_N ({\wedge}^n \mathcal{K}) \mathcal{P}_N) =  \frac{1}{n!} \sum_{\mathbf{m} \in \mathcal{M}^{(n)}_N} \langle \Phi_{N, \mathbf{m}}, ({\wedge}^n \mathcal{K}) \Phi_{N, \mathbf{m}} \rangle_{\wedge^n\emph{\textsf H}},
\end{equation}
is a sum over the much larger index set 
\begin{equation}
    \mathcal{M}_N^{(n)} = \{ \mathbf{m} = (m_1 \dots m_n) : 1 \leq m_{\alpha} \leq N, \forall \alpha \}.
\end{equation}
\end{lem}

\begin{proof} 
By the definition of the trace, 
\begin{equation}\label{trphiequality}
    \operatorname{Tr}(\mathcal{P}_N \mathcal{K} \mathcal{P}_N) = \!\!\! \sum_{\mathbf{m} \in \mathcal{M}_N^{(n), \uparrow}} \langle \Phi_{N,\mathbf{m}}, \wedge^n( \mathcal{K}) \Phi_{N,\mathbf{m}} \rangle_{\wedge^n\textsf{H}}.
\end{equation}
Let 
 \begin{equation}
     \mathcal{M}_N^{(n),0} = \{\mathbf{m} \in \mathcal{M}_N^{(n)} : \exists k \neq \ell : m_k = m_{\ell}\}.
 \end{equation}
Since $\Phi_{\mathbf{m}} = 0$ for all $\mathbf{m} \in\mathcal{M}_N^{(n),0} $ and
 \begin{equation}
     \mathcal{M}_N^{(n)} \setminus  \mathcal{M}_N^{(n),0} = \{\pi(\mathbf{m}) : \mathbf{m} \in \mathcal{M}_N^{(n), \uparrow}, \pi \in \sigma(n)\},
 \end{equation}
 we find that
 \begin{align}
     \sum_{\mathbf{m} \in \mathcal{M}_N^{(n)}} \langle \Phi_{N,\mathbf{m}}, {\wedge}^n (\mathcal{K}) \Phi_{N,\mathbf{m}} \rangle 
     &\,\,=\,\, \sum_{\mathbf{m} \in \mathcal{M}_N^{(n)} \setminus \mathcal{M}_N^{(n),0}} \langle \Phi_{N,\mathbf{m}}, {\wedge}^n (\mathcal{K}) \Phi_{N,\mathbf{m}} \rangle
     \nonumber
     \\
     &\,\,=\,\, \sum_{\mathbf{m} \in \mathcal{M}_N^{(n), \uparrow}} \, \sum_{\pi \in \sigma(n)} \langle \Phi_{N, \pi(\mathbf{m})}, {\wedge}^n (\mathcal{K}) \Phi_{N,\pi(\mathbf{m})} \rangle
    \nonumber
     \\
&\,\,=\,\, n! \sum_{\mathbf{m} \in \mathcal{M}_N^{(n),\uparrow}} \langle \Phi_{N, \mathbf{m}}, {\wedge}^n (\mathcal{K}) \Phi_{N, \mathbf{m}} \rangle,
 \end{align}
 as
$     \Phi_{ \pi(\mathbf{m})} = (-1)^{\pi} \Phi_{ \mathbf{m}}$.
  \end{proof} 
 
We conclude this section by generalizing a formula in Simon~\cite{Simon} for the 
$2$-modified Fredholm determinant of  a Hilbert-Schmidt operator, 
$\mathcal{K} \in \cB_2(L^2([a,b],\mathbb{C}))$, to the case of an operator, $\mathcal{K} \in \cB_2(L^2(\mathbb{R},\mathbb{C}^k))$. 
A similar result holds for operators,  
$\mathcal{K} \in \cB_2(L^2([a,b],\mathbb{C}^k))$, on a compact interval
without the need for \eqref{eq:HypExpDecay} in Hypothesis~\ref{hyp9}.

\begin{thm}\label{finaldet2thm}
Let $\mathcal{K} \in \cB_2(L^2(\mathbb{R}, \mathbb{C}^{k}))$ satisfy Hypothesis \ref{hyp9} and let
\begin{equation}\label{Ktildejn}
    \widetilde{K}_{\mathbf{j}}^{(n)}(\mathbf{x},\mathbf{y}) := \det\left[ \emph{\textsf{K}}_{j_{\alpha} j_{\beta}}(x_{\alpha}, y_{\beta}) [1 - \delta_{\alpha \beta}] \right]_{\alpha, \beta = 1}^{n},
\end{equation}
where $\delta_{\alpha \beta}$ is the Kronecker delta, and 
\begin{equation}\label{jinJ}
    \mathbf{j} \in J_k^{(n)} = \{ (j_1 \dots j_n) : 1 \leq j_{\alpha} \leq k, \,\, \forall \alpha \}.
\end{equation}
Then \begin{equation}\label{det2RCk}
    {\det}_2(\mathcal{I} + z\mathcal{K}) = \sum_{n=0}^{\infty} \frac{z^n}{n!} \sum_{\mathbf{j} \in J_k^{(n)}} \int_{\mathbb{R}^n} \widetilde{K}_{\mathbf{j}}^{(n)}(\mathbf{x}, \mathbf{x}) d\mathbf{x}.
\end{equation}
\end{thm}

\begin{proof}
    Since Hilbert-Schmidt operators are compact, $\mathcal{K}$ is the limit 
    in the Hilbert-Schmidt norm of finite rank operators which, in particular, are trace class.  
 So by the continuity of $\det_2$ given in (3) of Proposition~\ref{prop:det2}, it suffices to prove the result when $\cK$ is also trace class.
For each $n\in \mathbb N$ and  $ \mathbf{j} \in J^{(n)}_k$,  let 
\begin{equation}\label{defalpha}
    \alpha^{(n)}_{\mathbf{j}}(\lambda) := \int_{\mathbb{R}^n} {K}^{(n)}_{\mathbf{j}}(\mathbf{x},\mathbf{x}; \lambda) d\mathbf{x},
\end{equation}
where
\begin{equation}
    K^{(n)}_{\mathbf{j}}(\mathbf{x}, \mathbf{x}; \lambda) = \det\left[ \textsf{K}_{j_{\alpha} j_{\beta}}(x_{\alpha}, x_{\beta})(1 - \lambda \delta_{\alpha \beta}) \right]_{\alpha, \beta = 1}^n.
\end{equation}
Let $F : \mathbb{C} \rightarrow \mathbb{C} $ be defined by
\begin{equation}\label{eq:Fdef}
    F(\lambda) = \sum_{n=0}^{\infty} \,\,\sum_{\mathbf{j} \in J^{(n)}_k} \frac{\alpha^{(n)}_{\mathbf{j}}(\lambda)}{n!}.
\end{equation}
By (4) of Proposition~\ref{prop:det2}, it suffices to show that
   $ F(\lambda) = F(0) e^{-\lambda \operatorname{Tr}(\mathcal{K})}$,
or equivalently that
    $ F'(\lambda) = -\operatorname{Tr}(\mathcal{K}) F(\lambda)$.
 
 Let  $\mathbf{B}_{\mathbf{j}}$ be  the  $n \times n$ matrix defined by
 \begin{equation}\label{defB}
     [\mathbf{B}_{\mathbf{j}}(\lambda)]_{\alpha \beta} = 
     b_{\alpha\beta} := [\textsf{K}(x_{\alpha}, x_{\beta})]_{j_{\alpha} j_{\beta}}[1 - \lambda \delta_{\alpha \beta}].     
     \end{equation}
 Then by  the chain rule and the cofactor expansion for the determinant, 
 \begin{align}
     \frac{d}{d\lambda}({\det} \circ  \mathbf{B}_{\mathbf{j}})(\lambda) 
     &\,\,=\,\, \sum_{\alpha,\beta=1}^{n} \frac{\partial(\det \mathbf{B}_{\mathbf{j}})}{\partial b_{\alpha \beta}} \frac{\partial b_{\alpha \beta}}{\partial \lambda}
     \nonumber
     \\
&\,\,=\,\,     
       -\sum_{\alpha=1}^n \det\mathbf{B}_{\mathbf{j}}^{(\alpha)}\textsf{K}_{j_{\alpha}j_{\alpha}}(x_{\alpha},x_{\alpha}),
       \label{eq:lambdaderB}
 \end{align}
where $\mathbf{B}_{\mathbf{j}}^{(\alpha)}$ 
is the $(n-1) \times (n-1)$ matrix obtained by removing row $\alpha$ and column $\alpha$ from $\mathbf{B}_{\mathbf{j}}$. 
Next, we observe that for each $\alpha$, 
\begin{equation}\label{eq:detminorB}
\det \mathbf{B}_{\mathbf{j}}^{(\alpha)}
= K_{\mathbf{j}}^{(n-1)}(\mathbf{\widehat{x}}, \mathbf{\widehat{x}}; \lambda),
\quad\widehat{\mathbf x} = (x_1,\dots, x_{\alpha-1}, x_{\alpha+1}, x_n)\in \mathbb R^{n-1}.
\end{equation}
Since $\mathbf{B}_{\mathbf{j}}^{(\alpha)}$ does not include $x_{\alpha},$ and $\textsf{K}(x_{\alpha}, x_{\alpha})$ only involves $x_{\alpha}$, 
by \eqref{defalpha}, \eqref{defB}, \eqref{eq:lambdaderB}, and \eqref{eq:detminorB},
\begin{equation}\label{eq:finaldet2thmproof2} 
 \frac{\partial \alpha^{(n)}_{\mathbf{j}}}{\partial \lambda}(\lambda) = - \int_{\mathbb{R}^{n-1}} K_{\mathbf{j}}^{(n-1)}(\mathbf{x}, \mathbf{x}; \lambda) d\mathbf{x} \sum_{\alpha=1}^n \int_{\mathbb R}\textsf{K}_{j_{\alpha}j_{\alpha}}(x_\alpha,x_\alpha) dx_{\alpha}.
\end{equation}
Now, for   each $\alpha \in \{1, \dots, n\}, $ 
we have a bijection 
   $ P_{\alpha} \left(J^{(n-1)}_k \times J^{(1)}_k\right) \rightarrow J^{(n)}_k$,
between index sets given by
\begin{equation}
P_\alpha\left((j_1, j_2, \dots, j_{\alpha - 1}, j_{\alpha}, \dots, j_{n-1}), \beta \right) = (j_1, \dots, j_{\alpha - 1}, \beta, j_{\alpha}, \dots, j_{n-1}).
\end{equation}

Using this bijection together with ~\eqref{eq:Fdef}, \eqref{eq:finaldet2thmproof2}
and Theorem~\ref{thm:MatrixKernelTrace}
we have that
\begin{align} \nonumber
   F'(\lambda)
    &\,\,=\,\, - \sum_{n=0}^{\infty} \frac{1}{n!} \sum_{\alpha=1}^n \left( \sum_{\mathbf{j} \in J^{(n-1)}_{k}}   \alpha^{(n-1)}_{\mathbf{j}}(\lambda) \int_{\mathbb{R}} \sum_{j_{\alpha}=1}^k \textsf{K}_{j_{\alpha}j_{\alpha}}(x,x)dx \right) \\ \nonumber
    &\,\,=\,\,  - \sum_{n=0}^{\infty} \frac{1}{n!}  \left( \sum_{\mathbf{j} \in J^{(n-1)}_{k}} \alpha^{(n-1)}_{\mathbf{j}}(\lambda) \right) \sum_{\alpha=1}^n \int_{\mathbb{R}} \operatorname{Tr}(\textsf{K})(x,x)dx\\ \nonumber
&\,\,=\,\, -n \operatorname{Tr}(\mathcal{K}) \sum_{n=0}^{\infty} \frac{1}{n!} \sum_{\mathbf{j} \in J^{(n-1)}_{k}} \alpha^{(n-1)}_{\mathbf{j}}(\lambda)\\ \nonumber
&\,\,=\,\, -\operatorname{Tr}(\mathcal{K}) \sum_{n=1}^{\infty} \sum_{\mathbf{j} \in J^{(n-1)}_k} \frac{\alpha^{(n)}_{\mathbf{j}}(\lambda)}{(n-1)!}
\,\,=\,\, -\operatorname{Tr}(\mathcal{K}) F(\lambda),
\end{align}
 as required.
\end{proof}

\section{The von-Koch Formula for Block Matrices}\label{vonKochsection}

The classical von-Koch formula~\cite{vonkoch1892}, \cite[Section 3]{Bornemann},
which inspired Fredholm's definition of his determinant~\cite{Fred1903},
 states that if $\mathbf{K}$ is an $M\times M$ matrix, then 
\begin{equation}\label{VKF}
    \det(\mathbf{I} + z\mathbf{K}) = \sum_{n=0}^{\infty} \frac{z^n}{n!} \sum_{j_1, \dots, j_n = 1}^M \det\left( [{K}_{j_\alpha j_\beta}]_{\alpha,\beta=1}^n \right)
\end{equation}
can be expressed in terms of
determinants of submatrices of $\mathbf{K}$. This formula follows from the fact that 
$\operatorname{Tr}(\wedge^n \mathbf{K})$ can be expressed as the sum of \emph{all} 
 principal $n \times n$ minors of $\mathbf{K}$.
 In particular, the series in \eqref{VKF} terminates at $n=M$.
  Significantly for our purposes, 
Bornemann uses \eqref{VKF} in the derivation of a formula for a numerical approximation of the  Fredholm determinant of a trace class operator on $L^2([a,b], \mathbb{C}).$ In this section, we state two generalizations of the von-Koch formula to block matrices
 that we will use in Sections~\ref{truncsection} and \ref{quadsection} to obtain numerical approximations of the Fredholm determinants of 
 operators on $L^2(\mathbb{R},\mathbb{C}^k)$.

Let $I_M = \{1, 2, 3, \dots, M\}$ be an index set 
and let $L(I_M, \mathbb{C}^k)$ be the $kM$-dimensional complex vector space of functions,
 $ \phi : I_M \rightarrow \mathbb{C}^k$,
 endowed with the inner product
 \begin{equation}
    \langle \phi, \psi \rangle_{L(I_M, \mathbb{C}^k)} := \sum_{m=1}^M \langle \phi(m), \psi(m) \rangle_{\mathbb{C}^k}.
\end{equation}
The functions $\phi_{m, j} (\ell) := \delta_{m \ell} \mathbf{e}_j$
for $m=1,\dots,M$ and $j=1,\dots,k$ form an orthonormal basis for $L(I_M, \mathbb{C}^k)$.
Therefore, for any $\phi \in L(I_M, \mathbb{C}^k)$, 
\begin{equation}\label{disc6}
    \phi = \sum_{m=1}^M\sum_{j=1}^k [\phi(m)]_j \phi_{m, j}.
\end{equation}
Let $\mathcal{K}$ be a linear operator on $L(I_M, \mathbb{C}^k)$. 
For each $n,m\in\{1,\dots,M\}$, we let $\mathbf K(n,m)$ be the 
$k\times k$ matrix whose $(i,j)$-entry is given by
 \begin{equation}
 [\mathbf K(n,m)]_{i,j} = \langle \mathcal K \phi_{m,j}, \phi_{n,i}\rangle_{L(I_M,\mathbb C^k)}.
\end{equation}
The action of  $\mathcal{K}$ on $\phi$ is given in terms of a sum of matrix-vector products,
\begin{equation}\label{claimkphi}
    (\mathcal{K} \phi)(n) = \sum_{m=1}^M \mathbf{K}(n,m) \phi(m),
\end{equation}
which can be verified by taking the inner product of both sides with $\mathbf e_j$.
Assembling  the $k\times k$ matrices $\mathbf K(n,m)$ into the $M\times M$ block matrix
\begin{equation}\label{kblocks}
    \mathbf{K} = \begin{bmatrix}
    \mathbf{K}(1,1) & \dots & \mathbf{K}(1,M)\\
    \vdots & \quad & \vdots\\
    \mathbf{K}(M,1) & \dots & \mathbf{K}(M,M)
    \end{bmatrix} \in \mathbb{C}^{kM \times kM},
\end{equation}
we find that $\mathcal K \phi = \mathbf K \phi$, that is $\mathbf K$ is the matrix of $\mathcal K$. 

 If we regard $I_M$ as the index set for a discretization, $\{x_1, \dots, x_M\}$, of a finite interval $I$, then $L(I_M, \mathbb{C}^k)$  can be regarded as a discretization of
  $L^2(I, \mathbb{C}^k)$. 
 Imitating the  proofs of Theorems~\ref{prop:TraceIntegralFormulaScalarFinite0} and \ref{finaldet2thm}, with integrals over $\mathbb R^n$ replaced by sums over a multi-index,
 $\mathbf m$, we obtain the following two generalizations of the classical von Koch formula.
 The first of these determinental formulae is derived in \cite[(8.3)]{Bornemann} using a different approach.

\begin{thm}\label{det1IM}
Let $\mathcal{K}$ be a linear operator on $L(I_M, \mathbb{C}^k).$  Let
\begin{equation}
    \boldsymbol{\Phi}_{\mathbf{m}, \mathbf{j}} = \phi_{m_1, j_1} \wedge \dots \wedge \phi_{m_n, j_n} \in \wedge^n(L(I_M, \mathbb{C}^k))
\end{equation}
and let
\begin{eqnarray}
\mathcal{M}_M^{(n)} &=& \{\mathbf{m} = (m_1 \dots m_n) : 1 \leq m_{\alpha} \leq M \,\, \forall \alpha\},\\
J_k^{(n)} &=& \{\mathbf{j} = (j_1 \dots j_n) : 1 \leq j_{\alpha} \leq k \,\, \forall \alpha\}.
\end{eqnarray}
Then 
  \begin{equation}\label{discretetrace}
        \operatorname{Tr}(\wedge^n \mathcal{K}) = \frac{1}{n!}  \sum_{\mathbf{j} \in J_k^{(n)}}  \sum_{\mathbf{m} \in \mathcal{M}_M^{(n)}} 
 \langle \boldsymbol{\Phi}_{\mathbf{m}, \mathbf{j}}, (\wedge^n \mathcal{K}) \boldsymbol{\Phi}_{\mathbf{m}, \mathbf{j}} \rangle_{\wedge^n(L(I_M, \mathbb{C}^k))},
    \end{equation}
     and the regular Fredholm determinant,  ${\det}_1(\mathcal{I} + z \mathcal{K})$,  is given by
     \begin{equation}\label{discdet} 
          \det(\mathbf{I} + z \mathbf{K}) 
         = \sum_{n=0}^{\infty} \frac{z^n}{n!}  \sum_{\mathbf{j} \in J^{(n)}_k} \sum_{\mathbf{m} \in \mathcal{M}_M^{(n)}} \det\left( [\mathbf{K}(m_{\alpha}, m_{\beta})]_{j_{\alpha} j_{\beta}} \right)_{\alpha, \beta = 1}^n.
 \end{equation}
\end{thm}

Since all linear operators on finite-dimensional spaces are Hilbert-Schmidt, $\mathcal{K}$ has a $2$-modified Fredholm determinant, which is given by the following von-Koch formula for 
the block matrix, $\mathbf K$.  
 
\begin{thm}\label{vonKoch2block}
    Let $\mathbf{K}$ be the $M \times M$ 
 block matrix of $k \times k$ blocks,
 given by \eqref{kblocks}, and define
 \begin{equation}
     \widetilde{K}^{(n)}_{\mathbf{j}}(\mathbf{m}) := \det \left[ [\mathbf{K}(m_{\alpha}, m_{\beta})]_{j_{\alpha} j_{\beta}} (1 - \delta_{\alpha \beta})\right]_{\alpha, \beta = 1}^n
 \end{equation}
for each $\mathbf{j} \in J^{(n)}_k$ and $\mathbf{m} \in \mathcal{M}^{(n)}_M.$
Then the 2-modified Fredholm determinant, $ {\det}_{2}(\mathcal{I} + z\mathcal{K})$, is given by
\begin{equation}\label{det2matvk}
 {\det}_{2}(\mathbf{I} + z\mathbf{K}) = 
 e^{-\operatorname{Tr}(\mathbf K)}\,  {\det}(\mathbf{I} + z\mathbf{K}) =
 \sum_{n=0}^{\infty} \frac{z^n}{n!} \sum_{\mathbf{j} \in J^{(n)}_k}  \sum_{\mathbf{m}\in \mathcal{M}^{(n)}_{M}} \widetilde{K}^{(n)}_{\mathbf{j}}(\mathbf{m}).
 \end{equation}
\end{thm}
\begin{rem}
In Section~\ref{quadsection}, we will use 
  the similarities between the Fredholm 
  determinant formulae in \eqref{Kdetnbyn} and \eqref{det2RCk}
  and the generalized von Koch formulae in (\ref{discdet}) and \eqref{det2matvk}
  to approximate the regular and 2-modified Fredholm determinants of operators with matrix-valued kernels  by determinants of block matrices.
\end{rem}

\section{Domain Truncation Error}\label{truncsection}

Let $\mathcal K \in \mathcal B_p(L^2(\mathbb R, \mathbb C^k))$ be a trace class ($p=1$)
or Hilbert-Schmidt ($p=2$) operator with a matrix-valued kernel $\textsf{K}\in L^2(
\mathbb R\times \mathbb R, \mathbb C^{k\times k})$. Our goal is to numerically
approximate ${\det}_p(\mathcal I + z \mathcal K)$ by the determinant of a 
block matrix defined in terms of the kernel $\textsf K$. The approach we take is
based on that of Bornemann~\cite{Bornemann} for operators on a finite interval.
First, we truncate $\mathcal K$ to an operator on the finite interval $[-L,L]$.
Then use a quadrature rule to approximate the $n$-dimensional 
integrals in Theorems~\ref{prop:TraceIntegralFormulaScalarFinite0} and \ref{finaldet2thm}, 
which are now over  hypercubes $[-L,L]^n$, by $n$-fold sums.
Finally, using the von~Koch formulae in Theorems~\ref{det1IM} and \ref{vonKoch2block},
we obtain the 
desired approximation of the Fredholm determinant.
In this section, we obtain an estimate for the rate of convergence as 
$L\to \infty$ of the 
Fredholm determinant on the truncated interval, $[-L,L]$ to that on $\mathbb R$.  In Section~\ref{quadsection}, we derive an estimate for the
quadrature error as the discretization parameter, $\Delta x\to 0$. Finally, the total
error is bounded by the sum of these two errors.

    Let  $\mathcal{K} \in \cB_p(L^2(\mathbb{R},\mathbb{C}^k)),$ for $p=1$ or $2,$ and define 
    \begin{equation}
        \mathcal{K}|_{[-L,L]} := P_L \circ \mathcal{K} \circ \iota_{L},
    \end{equation}
    where $\iota_L : L^2([-L,L],\mathbb{C}^k) \rightarrow L^2(\mathbb{R},\mathbb{C}^k)$ is the inclusion operator and $P_L : L^2(\mathbb{R},\mathbb{C}^k) \rightarrow L^2([-L,L],\mathbb{C}^k)$ is the  projection operator given by
    \begin{equation}
        (P_L \boldsymbol{\psi})(x) = \bigchi_{[-L,L]}(x) \boldsymbol{\psi}(x).
    \end{equation}
    Here, $\bigchi_{[-L,L]}$ is the characteristic function of $[-L,L]$. Since 
    the operators, $\iota_L$ and $P_L$, are bounded and 
    $\cB_p$ is an ideal, $\mathcal{K}|_{[-L,L]} \in \cB_p(L^2([-L, L],\mathbb{C}^k)).$ 
 We note that if $\textsf{K}$ is the kernel for $\mathcal K$, then
\begin{equation}\label{QM1a}
    (\mathcal{K}|_{[-L,L]}\boldsymbol{\phi})(x) = \int_{-L}^L \textsf{K}(x,y)\boldsymbol{\phi}(y) dy.
\end{equation}

\begin{thm}\label{truncthm}
Let $\mathcal{K} \in \cB_p(L^2(\mathbb{R},\mathbb{C}^k))$, for $p=1$ or $2$, and suppose
that $\exists \, C, a > 0$ such that
\begin{equation}\label{truncexpass}
    |\emph{\textsf{K}}_{ij}(x,y)| \leq Ce^{-a(|x| + |y|)}, \,\, \forall i,j \in \{1,...,k\}, \forall x, y, \in \mathbb{R}.
\end{equation}
Then 
\begin{equation}\label{eq:truncthm}
    \left| {\det}_p(\mathcal{I} + z \mathcal{K}) - {\det}_p (\mathcal{I} + z \mathcal{K}|_{[-L,L]})\right| \leq e^{-aL} \boldsymbol{\Phi} \left( \frac{2Ckz}{a} \right),
\end{equation}
where 
\begin{equation}\label{phisum}
    \boldsymbol{\Phi}(z) = \sum_{n=1}^{\infty} \frac{n^{(n+2)/2}}{n!} z^n.
\end{equation}
\end{thm}

\begin{rem}
Bornemann \cite{Bornemann} shows that
$\boldsymbol{\Phi}(z) \leq z \boldsymbol{\Psi}(z\sqrt{2} e)$,
where
\begin{equation}
    \boldsymbol{\Psi}(z) = 1 + \frac{\sqrt{\pi}}{2}z e^{z^2/4} \left[1 + \erf\left(\frac{z}{2}\right) \right].
\end{equation}
Consequently, the upper bound in \eqref{eq:truncthm} can be estimated by a 
computable constant depending on the decay of the kernel $\textsf{K}.$
\end{rem}

\begin{proof}
Let
\begin{equation}\label{QM8a}
    \widetilde{K}^{(n)}_{\mathbf{j}}(\mathbf{x}, \mathbf{y}) := {\det}\left[ \textsf K_{j_{\alpha}j_{\beta}}(x_{\alpha}, y_{\beta}) (1 - (p-1)\delta_{\alpha \beta}) \right]_{\alpha, \beta = 1}^{n}.
\end{equation}
By Theorems \ref{prop:TraceIntegralFormulaScalarFinite0} and \ref{finaldet2thm}, the error in \eqref{eq:truncthm}  is given by 
\begin{align}\label{seriesL}
    E_{L} &\,\,:=\,\, |{\det}_p(\mathcal{I} + z\mathcal{K}) - {\det}_p(\mathcal{I} + z \mathcal{K}|_{[-L,L]})| 
    \nonumber
 \\    
   & \,\,\leq\,\,  \sum_{n=1}^{\infty} \frac{z^n}{n!} \sum_{\mathbf{j} \in J^{(n)}_k} \int_{\mathbb{R}^n \setminus [-L,L]^n} |\widetilde{K}^{(n)}_{\mathbf{j}}(\mathbf{x}, \mathbf{x})| d\mathbf{x}.
\end{align}
Let 
$\mathbf{A}_{*\beta}$ denote column $\beta$ of the matrix in \eqref{QM8a} whose determinant 
is  $\widetilde{K}^{(n)}_{\mathbf{j}}(\mathbf{x}, \mathbf{x})$.
By assumption (\ref{truncexpass}), 
\begin{equation} 
\|\mathbf{A}_{*\beta}\|_2 ^2  
\leq C^2 e^{-2 a |x_{\beta}|} \left[ e^{-2 a|x_1|} + \cdots + e^{-2 a|x_n|} \right]
\leq n C^2 e^{-2 a  |x_{\beta}|},
\end{equation}
where the second inequality follows from the fact that $e^{-a|x|} < 1$ for $a > 0.$
Therefore, by Hadamard's inequality
\begin{equation}\label{ktildehad}
    |\widetilde{K}^{(n)}_{\mathbf{j}}(\mathbf{x})| \leq C^n n^{n/2} e^{- a  \|\mathbf{x}\|_1},
\end{equation}
and so, since $|J^{(n)}_k| = k^n$,  
\begin{equation}\label{EL}
    |E_{L}| \leq \sum_{n=1}^{\infty} \frac{z^n}{n!} (Ck)^n n^{n/2} \int_{\mathbb{R}^n \setminus [-L,L])^n} e^{-a \|\mathbf{x}\|_1} d\mathbf{x}.
\end{equation}

When $n=2$, the integral in \eqref{EL} can be estimated as follows. First,
\begin{equation}
    \mathbb{R}^2 \setminus [-L,L]^2 = \{ |x_1| > L\} \cup \{|x_2|>L\},
\end{equation} although this union is not disjoint. Therefore,
\begin{align}
    &\int\limits_{\mathbb{R}^2 \setminus [-L,L]^2} 
     e^{-a(|x_1| + |x_2|)} dx_1 dx_2 
     \nonumber
     \\
     &\,\,\leq\,\,
   \int\limits_{|x_1|>L} dx_1  \int_{-\infty}^{\infty} e^{-a (|x_1| + |x_2|)} dx_2 +
   \int\limits_{|x_2|>L}dx_2\! \int_{-\infty}^{\infty}e^{-a (|x_1| + |x_2|)} dx_1
   \nonumber
   \\
   &\,\, =\,\, 2 \int_{|x|>L} e^{-a|x|}dx   \int_{-\infty}^\infty e^{-a|x|}
   \,\,=\,\,  2 \left( \frac{2}{a} e^{-a L}\right)\left( \frac{2}{a} \right).
\end{align}
Similarly, 
$\int_{\mathbb{R}^n\setminus [-L,L]^n} e^{-a \|\mathbf{x}\|_1} d\mathbf{x} \leq n \left(\frac{2}{a}\right)^n e^{-a L}$,
and so by (\ref{ktildehad}), 
\begin{equation} 
    |E_L| \,\,\leq\,\, \sum_{n=1}^{\infty} \frac{z^n}{n!} \left(\frac{2Ck}{a}\right)^n n^{(n+2)/2} e^{-aL}    \,\,=\,\, e^{- a L} \boldsymbol{\Phi}\left( \frac{2zCk}{a}\right).
\end{equation}
\end{proof}

\section{Quadrature Approximation and Error}\label{quadsection}

In order to compute $\det_p(\mathcal{I} + z\mathcal{K}|_{[-L, L]}),$ for $p=1$ or $2,$ we approximate the integrals $\int_{[-L,L]^n} \widetilde{K}^{(n)}_{\mathbf{j}}(\mathbf x, \mathbf x) d\mathbf{x}$ using a numerical integration scheme. 
This approach enables us to approximate the Fredholm determinant by a 
matrix determinant.
 In Theorem \ref{quadphibound} below, we show that if we use a quadrature rule based on the composite Simpson's rule and 
assume that the kernel $\textsf{K}$ is Lipschitz-continuous,  then the error
between these two determinants is ${O}(\Delta x),$ where $\Delta x$ is the grid spacing
of the quadrature rule.

For scalar-valued functions, $f :[a,b]\to\mathbb{C}$, let
$Q_M$ be a quadrature rule   of the form~\cite{atkinson2008introduction}
\begin{equation}
    Q_M(f) = \sum_{i=1}^M w_i f(x_i),
\end{equation}
that is defined in terms of $M$ nodes $a \leq x_1 < x_2 < \dots < x_M \leq b$ and positive weights $w_1, \dots, w_M.$ We suppose that this family of quadrature rules
converges for continuous functions in that 
\begin{equation}
    Q_M(f) \to \int_a^b f(x)dx, \,\, \text{ as } M \to \infty,
\end{equation}
for all $f \in C^0([a,b],\mathbb{C})$.

Let $\mathcal{K} \in \cB_p(L^2([a,b], \mathbb{C}^k))$ be a trace class ($p=1$) 
or Hilbert-Schmidt ($p=2$) operator with a matrix-valued kernel $\textsf{K} \in C^0 ([a,b] \times [a,b], \mathbb{C}^{k \times k})$, and let 
\begin{equation}\label{QM2}
    d_p(z) := {\det}_p(\mathcal{I} + z\mathcal{K}).
\end{equation}
Let $\mathbf{K}_Q \in \mathbb{C}^{kM \times kM}$  be the $M \times M$ block matrix 
\begin{equation}\label{QM4}
    \mathbf{K}_Q = \begin{bmatrix} w_1\mathbf{K}(1,1) & w_2 \mathbf{K}(1,2) & \dots & w_M \mathbf{K}(1,M)\\
    w_1 \mathbf{K}(2,1) & w_2 \mathbf{K}(2,2) & \dots & w_M \mathbf{K}(2,M)\\
    \vdots & & & \vdots\\
    w_1 \mathbf{K}(M,1) & w_2 \mathbf{K}(M,2) & \dots & w_M \mathbf{K}(M,M)
    \end{bmatrix},
\end{equation}
where 
   $ \mathbf{K}(\alpha, \beta) := \textsf{K}(x_{\alpha}, x_{\beta}) \in \mathbb{C}^{k \times k}$
is the $k \times k$ matrix obtained by evaluating $\textsf K$ at the nodes  $x_{\alpha}, x_{\beta} \in \{x_i\}_{i=1}^M$ of the quadrature rule $Q_M$.
Then the matrix determinant approximations of the Fredholm determinants in 
 \eqref{QM2} are defined by
\begin{align}\label{}
    d_{1,Q_M}(x) \,\,&=\,\, \det[\mathbf{I}_{kM \times kM} + z \mathbf{K}_{Q}],
\\
    d_{2,Q_M}(z) \,\,&=\,\, {\det}[\mathbf{I}_{kM \times kM} + z\mathbf{K}_Q]e^{-z \operatorname{Tr}(\mathbf{K}_Q)}.
\end{align}

As in  \cite[Theorems 6.1 \& 8.1]{Bornemann},  we have the following result.
\begin{thm}
Suppose that $\mathcal{K} \in \cB_p(L^2([a,b], \mathbb{C}^k))$ is a trace class or Hilbert-Schmidt operator with continuous matrix-valued kernel $\emph{\textsf{K}}$ and that $Q_M$ is a family of quadrature rules on $[a,b]$ that converges for continuous functions. Then
\begin{equation}
    d_{p,Q_M}(z) \to d_p(z) \,\, \text{ as } M \to \infty
\end{equation}
uniformly in $z$.
\end{thm}

\begin{proof}
Recall that 
\begin{equation}\label{QM9}
    {\det}_p(\mathcal{I} + z\mathcal{K})
    = 1 + \sum_{n=1}^{\infty} \frac{z^n}{n!} \sum_{\mathbf{j} \in J^{(n)}_k} \int_{[a,b]^n} \widetilde{K}^{(n)}_{\mathbf{j}}(\mathbf{x}) d\mathbf{x},
\end{equation}
where $\widetilde{K}^{(n)}_{\mathbf{j}}$ is given by \eqref{QM8a}. 
To approximate the integral in \eqref{QM9}, we recall that 
 the 1-dimensional quadrature rule, $Q_M,$  defines an $n$-dimensional quadrature rule,
\begin{equation}\label{QM10}
    Q^{(n)}_M (f) := \sum_{\mathbf{m} \in \mathcal{M}^{(n)}_M} w_{m_1} \dots w_{m_M} f(x_{m_1}, \dots, x_{m_M}),
\end{equation} 
which converges for continuous functions. 
Calculating formally without regard for convergence issues,  by (\ref{QM10}), (\ref{QM8a}), the multilinearity property of determinants, and (\ref{QM4}), we have that for large $M$,
\begin{align}
&d_p(z)  
\approx 1 + \sum_{n=1}^{\infty}\frac{z^n}{n!} \sum_{\mathbf{j}\in J^{(n)}_k} Q^{(n)}_M 
( {\widetilde K}^{(n)}_{\mathbf j})
\nonumber
 \\
&= 1 + \sum_{n=1}^{\infty}\frac{z^n}{n!} \sum_{\mathbf{j},\mathbf {m}} w_{m_1} \dots  w_{m_M} {\det}\left[  \mathbf{K}_{j_{\alpha} j_{\beta}}(m_{\alpha}, m_{\beta})(1 - (p-1)\delta_{\alpha \beta}) \right]_{\alpha, \beta = 1}^{n}
\nonumber
\\
&=1 + \sum_{n=1}^{\infty}\frac{z^n}{n!} \sum_{\mathbf{j},\mathbf {m}}  {\det}\left[  [w_{m_\alpha}\mathbf{K}(m_{\alpha}, m_{\beta}]_{j_{\alpha} j_{\beta}}(1 - (p-1)\delta_{\alpha \beta}) \right]_{\alpha, \beta = 1}^{n}
\nonumber
\\
&=d_{p,Q_M}(z),
\label{eq:dpquadderivation}
\end{align}
where the inner sum is taken over $J^{(n)}_k\times  \mathcal{M}^{(n)}_M$
and  the final equality holds by the  von-Koch formulae \eqref{discdet} and \eqref{det2matvk}. 
The estimates required to obtain uniform convergence are very similar to those 
 for scalar-valued kernels given in the proof of Theorem~6.1 of Bornemann~\cite{Bornemann}.
\end{proof} 

For the purposes of numerical computation, it is more useful to obtain an estimate
on the rate of convergence of $d_{p,Q_M}$ to $d_p$ as the maximum grid spacing
converges to zero. To obtain such a result we need some additional regularity on the kernel.

Let  $C^{r,1}([a,b], \mathbb{C})$ be the space of functions whose  
$r$-th derivative, $f^{(r)}$, is Lipschitz continuous.  Such
functions are differentiable almost everywhere with bounded derivative,
\begin{equation}
    \|f^{(r)}\|_{L^{\infty}([a,b])} := \sup_{x \in [a,b]} |f^{(r)}(x)| < \infty.
\end{equation}
Furthermore, for a function, $f\in C^{r,1}([a,b]^n, \mathbb{C})$, of $n$ variables,
we have 
\begin{equation}
    |f|_r := \sum\limits_{i=1}^n \| \partial_i^r f \|_{L^{\infty}([a,b]^n)}   < \infty.
\end{equation}

Recall that a one-dimensional quadrature rule, $Q$, is of order $\nu_Q$ if $Q(f)$ is exact for all polynomials $f$ up to degree $\nu_Q - 1$~\cite{atkinson2008introduction}.
By Theorem A.3 of Bornemann~\cite{Bornemann}, if 
$f \in \mathcal{C}^{r-1,1}([a,b]^n,\mathbb C),$ then for each quadrature rule $Q$ of order $\nu_Q \geq p$ with positive weights,\footnote{Bornemann omits the  factor of $n$, whose
necessity is evident in the proof of Theorem~\ref{thmcompbound} below, and in Davis and Rabinowitz~\cite[p. 361]{DavisRabinowitz}.}
\begin{equation}\label{Bornquadn}
    \left| Q^{(n)}(f) - \int_{[a,b]^n} f(\mathbf x) d\mathbf x \right| \leq c_r \nu_Q^{-r} 
    |f|_r  n (b-a)^{n+r},
\end{equation}
where the constant, $c_r=2\left(\frac{\pi e}{4}\right)^r / \sqrt{2 \pi r}$, only depends on $r$.
This result shows that the quadrature error decays like $\nu_Q^{-r}$ as $\nu_Q \rightarrow \infty$. This estimate is useful in situations where $f$ is approximated by a single high-degree polynomial on the entire interval $[a,b].$ However,  it is often more practical to 
employ a composite quadrature rule in which the domain $[a,b]$ is divided into 
subintervals and $f$ is separately approximated  by a low degree polynomial on each subinterval.
For such a composite rule, we are interested in the rate at which the quadrature error converges to zero as the maximum grid spacing converges to zero.  Bornemann's error estimate (\ref{Bornquadn}) cannot directly be applied to the $n$-dimensional  extension of a 
1-dimensional composite quadrature rule. Instead, here we  derive an error estimate for the $n$-dimensional quadrature rule that is based on the composite Simpson's rule, which is of order $\nu_Q = 4$.

We divide the interval $[a,b]$ into $M$ 
subintervals $I_j$ of width  $\Delta x_j$. 
To allow for an adaptive method we allow the $\Delta x_j$ to vary with $j$.
We let  nodes, $a = x_1 < x_2 < \dots < x_{2M+1} = b$, be such that
 $I_j = [x_{2j-1}, x_{2j+1}]$ and $x_{2j}$ is the midpoint of $I_j$ for $j=1,\dots, M$.
On each subinterval, $I_j,$ we apply  Simpson's rule, 
\begin{equation}
    Q_{I_j}(f) := \sum_{k=-1}^1 w_{j,k} f(x_{2j-k}) \approx \int_{I_j} f(x) dx,
\end{equation}
where $w_{j,-1} = w_{j,1} = \frac{\Delta x_j}{6},$ and $w_{j,0} = \frac{2 \Delta x_j}{3}.$
By \eqref{Bornquadn}, if $f \in C^{r-1,1}([a,b], \mathbb{C}),$ then 
\begin{equation}\label{1dintsimpbound}
    \left| Q_{I_j}(f) - \int_{I_j} f(x) dx \right| \leq c_r 4^{-r} \|f^{(r)}\|_{L^{\infty}(I_j)}\Delta x_j (\Delta x_{\operatorname{max}})^{r} 
\end{equation}
where
$\Delta x_{\operatorname{max}}= \operatorname{max}_j  \Delta x_j$.
Then the composite Simpson's rule is given by 
\begin{equation}\label{qdoublestar}
    Q_{[a,b]}(f) := \sum_{j=1}^M Q_{I_j}(f) \approx \int_a^b f(x) dx,
\end{equation}  and by (\ref{1dintsimpbound}), 
\begin{equation}
\left| Q_{[a,b]}(f) - \int_{a}^b f(x) dx \right| \leq c_r  4^{-r} \|f^{(r)}\|_{L^{\infty}[a,b]}
 (b-a)(\Delta x_{\operatorname{max}})^r.
\end{equation}

We observe that
\begin{equation}\label{Qab}
    Q_{[a,b]}(f) = \sum_{k=1}^{2M+1} w_k f(x_k),
\end{equation}
where now
$w_1 = \frac{\Delta x_1}{6}$,  $w_{2M+1} = \frac{\Delta x_M}{6}$, 
$w_{2j} = \frac{2 \Delta x_j}{3}$, and $w_{2j+1} = \frac{(\Delta x_j + \Delta x_{j+1})}{6}$ for $j=1, \dots, M-1.$

\begin{thm}\label{thmcompbound}
Suppose that $f \in C^{r-1,1}([a,b]^n,\mathbb C)$ for some $1\leq r \leq 4.$ Let $Q_{[a,b]}^{(n)}$ be the $n$-dimensional quadrature rule induced by the composite Simpson's quadrature rule $Q_{[a,b]}$ in (\ref{Qab}). Then
\begin{equation}
 \left| Q_{[a,b]}^{(n)}(f) - \int_{[a,b]^n} f(\mathbf{x}) d\mathbf{x}  \right| \leq c_r  4^{-r}  |f|_r n(b-a)^n  (\Delta x_{\operatorname{max}})^r.
\end{equation}
\end{thm}
\begin{proof}
Since 
   $ [a,b]^n = \bigcup_{j_1,\dots,j_n=1}^M I_{j_1} \times \dots \times I_{j_n}$,
\begin{equation}\label{Qnfab}
    Q_{[a,b]}^{(n)}(f) = \sum_{j_1, \dots, j_n=1}^M Q_{j_1 \dots j_n}(f),
\end{equation}
where
\begin{eqnarray} \nonumber 
    Q_{j_1 \dots j_n}(f) &=& \sum_{k_1, \dots, k_n = -1}^{1} w_{j_1,k_1} \dots w_{j_n,k_n} f(x_{2j_1 - k_1}, \dots, x_{2j_n - k_n})\\  
    &\cong& \int_{I_{j_1}\times \dots \times I_{j_n}} f(x_1,\dots, x_n) dx_1 \dots dx_n.
\end{eqnarray}
For each $m \leq n$ let 
\begin{equation}
    F_{j_1 \dots j_m}(y_{m+1}, \dots, y_n) := \int_{I_{j_1} \times \dots \times I_{j_m}} f(x_1, \dots, x_m, y_{m+1}, \dots, y_n) dx_1 \dots dx_m
\end{equation}
be the partial iterated integral of $f,$ which is approximated by the partial quadrature rule 
\begin{align}
    &(Q_{j_1 \dots j_m} f)(y_{m+1}, \dots, y_n)  
    \nonumber
    \\
    &= \sum_{k_1, \dots, k_m = -1}^1  w_{k_1}\dots w_{k_m} f(x_{2j_1 - k_1}, \dots, x_{2j_m - k_m}, y_{m+1}, \dots, y_n).  
\end{align}
The identity $Q_{j_1 \dots j_n}(f)  = Q_{j_n} (Q_{j_1 \dots j_{n-1}} f)$
enables us to use  induction to prove the following claim. 

\begin{claim}
\begin{equation}\label{anclaim}
    \left|  Q_{j_1 \dots j_n}(f) - \int\limits_{I_{j_1} \times \dots \times I_{j_n}}  f(\mathbf x) d\mathbf{x}   \right| 
     \,\,\leq\,\,
     c_r 4^{-r}|f|_r \,n     (\Delta x_{\operatorname{max}})^{r}  
     \prod\limits_{k=1}^n \Delta x_{j_k}
\end{equation}
\end{claim}

\begin{proof} 
We calculate that 
\begin{align} 
&\left| Q_{j_1 \dots j_n}(f) \,\,-\,\, \int_{I_{j_1 \times \dots \times I_{j_n}}}   f(\mathbf{x}) d\mathbf{x}  \right| 
 \nonumber
\\
\,\,=\,\,& \left| 
Q_{j_n}(Q_{j_1,\dots,j_{n-1}} f) \,\,-\,\,  \int_{I_{j_n}}\int_{I_{j_1} \times \dots \times I_{j_{n-1}}}  f(\mathbf{x}, x_n) d\mathbf{x} \,dx_n  \right|
 \nonumber
\\
\,\,\leq\,\, & \left| Q_{I_{j_n}} (Q_{j_1,\dots,j_{n-1}} f) \,\,-\,\, \int_{I_{j_n}} (Q_{j_1 \dots j_{n-1}} f)(x_n) dx_n  \right|
 \nonumber 
\\
&\,\,+\,\,   \int_{I_{j_n}} \left| (Q_{j_1,\dots,j_{n-1}} f)(x_n)\,\, -\,\, 
\int_{I_{j_1} \times \dots \times I_{j_{n-1}}}  f(\mathbf{x}, x_n) d\mathbf{x}  \right| dx_n
 \nonumber 
\\  
\,\,\leq\,\, & c_r 4^{-r} (\Delta x_{\operatorname{max}})^{r}  
\left( | Q_{j_1 \dots j_{n-1}} f |_r \Delta x_{j_n}  
+  |f|_r (n-1)     \prod\limits_{k=1}^{n} \Delta x_k   \right)
\nonumber 
\\
\,\,\leq\,\,&   c_r 4^{-r}|f|_r \,n     (\Delta x_{\operatorname{max}})^{r}  
     \prod\limits_{k=1}^n \Delta x_{j_k},
     \nonumber
 \end{align}
since 
\begin{align}
&\left| (Q_{j_1 \dots j_{n-1}} f)^{(r)}(x_n) \right| 
\nonumber
\\
\,\,\leq\,\,&
\sum_{k_1,\dots,k_{n-1} = -1}^1 \!\!\!\!\! \!\!\! w_{j_1,k_1} \dots w_{j_{n-1},k_{n-1}} \left| \partial_n^r f(x_{2j_1 - k_1}, \dots, x_{2j_{n-1} - k_{n-1}}, x_n) \right|
\nonumber 
\\ 
\,\,\leq\,\,&  |f|_r  \prod\limits_{k=1}^{n-1}(w_{j_k,-1} + w_{j_k,0} + w_{j_k,1})
\,\,\leq\,\,  |f|_r  \prod\limits_{k=1}^{n-1} \Delta x_{j_k}.
\nonumber
\end{align}
\end{proof}
To complete the proof of the Theorem, by (\ref{Qnfab}) and (\ref{anclaim}), we have
\begin{align}
\left | Q_{[a,b]}^{(n)}f  - \int_{[a,b]^n}  f(\mathbf{x}) d\mathbf{x} \right|
&\,\,\leq\,\, 
\sum\limits_{j_1 \dots j_n = 1}^M \left| 
Q_{j_1 \dots j_n} f  \,\,- 
\int\limits_{I_{j_1}\times \dots \times I_{j_n}} f(\mathbf{x}) d\mathbf{x} 
\right|
\nonumber
\\
&\,\,\leq\,\, 
c_r 4^{-r}|f|_r \,n     (\Delta x_{\operatorname{max}})^{r}  
   \sum\limits_{j_1 \dots j_n = 1}^M   \prod\limits_{k=1}^n \Delta x_{j_k},
\end{align}
where
$ \sum\limits_{j_1 \dots j_n = 1}^M   \prod\limits_{k=1}^n \Delta x_{j_k} =
(b-a)  \sum\limits_{j_1 \dots j_{n-1} = 1}^M   \prod\limits_{k=1}^{n-1} \Delta x_{j_k}
= (b-a)^n$.
\end{proof}

We can now state the main result of this section.

\begin{thm}\label{quadphibound}
Let $\mathcal{K}\in \cB_p(L^2([a,b],\mathbb{C}^k))$ be an operator with matrix-valued kernel $\emph{\textsf{K}} \in L^2([a,b]\times[a,b], \mathbb{C}^{k \times k}).$ Suppose that for some $1\leq r\leq 4$,
$\emph{\textsf{K}}_{ij} \in C^{r-1,1}([a,b] \times [a,b], \mathbb{C})$ for all $i,j \in \{1, \dots, k\}$.
Let $Q$ be an adaptive composite Simpson's quadrature rule on $[a,b]$ with maximum grid 
spacing $\Delta x_{\operatorname{max}}$. Then
\begin{equation}\label{quaderrest}
    |d_{p,Q}(z) - d_p(z)| \leq \frac{c_r}{2^r}  \boldsymbol{\Phi}\left(2k(b-a)\|\emph{\textsf{K}}\|_{r} |z|\right) (\Delta x_{\operatorname{max}})^r,
\end{equation}
where  $c_r$ is the constant in Theorem \ref{thmcompbound}  $ \boldsymbol{\Phi}$ is defined in \eqref{phisum}, and
\begin{equation}\label{newnorm}
\|\emph{\textsf{K}}\|_{r} = \max\limits_{i+j\leq r} \| \partial^i_x\partial^j_y \emph{\textsf K} \|_{L^\infty([a,b]^2)}.\end{equation}

\end{thm}

\begin{proof}
As in Lemma A.4 of \cite{Bornemann}, 
\begin{equation}
    \left| \widetilde{K}^{(n)}_{\mathbf{j}} \right|_r 
     \leq 2^r n^{(n+2)/2} \|\textsf{K}\|_{r}^n
\end{equation}
where $\widetilde{K}^{(n)}_{\mathbf{j}}$ is given by (\ref{QM8a}). 
So, by \eqref{eq:dpquadderivation} and Theorem~\ref{thmcompbound},
\begin{align}
|d_{p,Q}(z) - d_p(z)| 
&\,\,\leq\,\, 
 \sum\limits_{n=1}^{\infty} \frac{|z|^n}{n!} \sum\limits_{\mathbf{j} \in J^{(n)}_k} \left| Q^{(n)} (\widetilde{K}^{(n)}_{\mathbf{j}}) - \int_{[a,b]^n}\widetilde{K}^{(n)}_{\mathbf{j}}(\mathbf x) 
 d\mathbf x \right|
 \nonumber 
\\
&\,\,\leq\,\, \frac{c_r}{4^r}  (\Delta x_{\operatorname{max}})^r \sum_{n=1}^{\infty} \frac{|z|^n}{n!} 
n (b-a)^n \sum_{\mathbf{j} \in J^{(n)}_k} \left| \widetilde{K}^{(n)}_{\mathbf{j}} \right|_r
 \nonumber 
\\
&\,\,\leq\,\, \frac{c_r  }{2^r} ( \Delta x_{\operatorname{max}})^r
\sum_{n=1}^{\infty} n\,\frac{n^{(n+2)/2}}{n!} (|z|(b-a)k\|\textsf{K}\|_{r})^n
\nonumber
\\
&\,\,\leq\,\, \frac{c_r }{2^r} (\Delta x_{\operatorname{max}})^r \boldsymbol{\Phi}(2k(b-a)\|\textsf{K}\|_{r}|z|)),
\end{align}
as $|J^{(n)}_k| = k^n$ and $n \leq 2^n$.

\end{proof}

\section{Numerical Results for Birman-Schwinger Operators}\label{sec:Results}

In this section, we briefly summarize the construction of the Birman-Schwinger operator 
$\mathcal K(\lambda)$ discussed in the introduction that characterizes the stability of stationary pulse solutions
of certain nonlinear waves equations. In particular, we 
state the formula for the matrix-valued kernel of this operator in the special case of the hyperbolic secant solution of the nonlinear Schr\"odinger equation. 
Since this  kernel  is Lipschitz continuous and decays exponentially, we can apply the trace class regularity theorem we proved in~\cite{zweck2024regularity}, to show that the Hilbert-Schmidt operator $\mathcal K$ 
is also trace class. Consequently, by~\cite{EJF}, 
$\operatorname{det}_1 (\mathcal I+\mathcal K(\lambda)) = E(\lambda)$, where
$E(\lambda)$ is the Evans function. 
Since  $\lambda=0$ is an eigenvalue of the linearized
operator, we know that $\operatorname{det}_1 (\mathcal I+\mathcal K(0)) = 0$. 
We then use this fact to verify Theorems~\ref{truncthm} and \ref{thmcompbound}
on the rate of convergence of the numerical approximation 
$d_{p,Q}$, to the Fredholm determinant, $d_p$, as $L\to\infty$ and 
$\Delta x_{\operatorname{max}}\to 0$.  

We consider solutions of the nonlinear Schr\"odinger equation,
\begin{equation}
i\psi_t + \frac D2 \psi_{xx} + \gamma |\psi |^2 \psi = 0,
\end{equation} 
with $D\neq 0$, $\gamma > 0$, 
 that are of the form
   $ \psi(t,x) = \Psi(x)e^{-i \alpha t}$,
for some phase change $\alpha\in \mathbb R$. Then $\boldsymbol{\Psi} := [\text{Re}({\Psi}) \,\,\, \text{Im}({\Psi})]^T$
satisfies 
\begin{equation}\label{cq3}
    \partial_t \boldsymbol{\Psi} = \left( \mathbf{B}\partial_x^2 + \mathbf{N}_0 + \mathbf{N}_1|\boldsymbol{\Psi}|^2 \right) \boldsymbol{\Psi},
\end{equation}
where $\boldsymbol{\Psi} : \mathbb R \to \mathbb R^2$, and
\begin{equation}
    \mathbf{B} = 
    \begin{bmatrix}
    0 & -\frac{D}{2}\\
    \frac{D}{2} & 0
    \end{bmatrix},
    \,\,\,\,
\mathbf{N}_0 = 
\begin{bmatrix}
0 & -\alpha\\ \alpha & 0
\end{bmatrix},
 \,\,\,\,
\mathbf{N}_1 = 
\begin{bmatrix}
0 & -\gamma\\ \gamma & 0
\end{bmatrix}.
\end{equation}
Linearizing (\ref{cq3}) about a stationary solution $\boldsymbol{\Psi},$ we obtain the equation 
\begin{equation}\label{linbasic}
    \partial_t \mathbf{p} = \mathcal{L}\mathbf{p},
\end{equation}
where  
\begin{equation}\label{operatorL}
    \mathcal{L}= \mathbf{B}\partial_x^2 +  \mathbf{N}_0 +  {\mathbf{M}}(x) 
\end{equation}
with
\begin{equation}\label{MTilde}
   \mathbf{M}(x) = |\Psi|^2\mathbf{N}_1  +  2 |\Psi|^2 \mathbf{N}_1\Psi \Psi^T.
\end{equation}
In  \eqref{linbasic}, we regard $\mathbf p$ as a mapping, $\mathbf {p} : \mathbb R \to \mathbb C^2$, since the spectrum of the non-self adjoint operator, $\mathcal L$, is not constrained to be real.  

Let $\mathbf Y=\mathbf Y(x) \in \mathbb C^4$ be given by 
$\mathbf Y^T= \begin{bmatrix} \mathbf p^T & \mathbf p_x^T \end{bmatrix}$, and 
for a complex parameter, $\lambda$, let
\begin{equation}\label{A}
     \mathbf{A}_\infty(\lambda) = \begin{bmatrix}
     0 & \mathbf{I}\\ \mathbf{B}^{-1}(\lambda - \mathbf{N}_0) & 0
     \end{bmatrix} 
     \quad \text{and} \quad
        \mathbf{R}(x) = \begin{bmatrix}
    \mathbf{0} & \mathbf{0}\\ -\mathbf{B}^{-1}\mathbf{M}(x) & \mathbf{0}
    \end{bmatrix}.
 \end{equation}
Then  $\lambda \in \{ \lambda \in \mathbb C \,:\,\operatorname{Spec}(\mathbf A_\infty(\lambda) \cap i \mathbb R = \emptyset\}$ lies in the point spectrum of 
$\mathcal{L}$ if and only if the system
\begin{equation}\label{systempert}
    \partial_x \mathbf{Y} = [\mathbf{A}_{\infty}(\lambda) + \mathbf{R}(x)]\mathbf{Y}
\end{equation}
has a bounded solution. 
By the polar decomposition,  
$\mathbf R=\mathbf R_\ell \mathbf R_r$, where
\begin{equation}\label{Rrnorm}
    \mathbf{R}_r = \begin{bmatrix}
    |\mathbf{M}|^{1/2} & \mathbf{0}\\ \mathbf{0} & \mathbf{0}
    \end{bmatrix}
    \quad
    \text{and}
    \quad
    \mathbf{R}_{\ell}  = \begin{bmatrix}
    \mathbf{0} & \mathbf{0}\\ -\mathbf{B}^{-1}\mathbf{M}|\mathbf{M}|^{-1/2} & \mathbf{0}
    \end{bmatrix}.
\end{equation}
Let $\textsf K$ be the kernel 
\begin{equation}
\nonumber
    \textsf{K}(x,y;\lambda) = \begin{cases}
    -\mathbf{R}_{r}(x)\mathbf{Q}(\lambda)e^{\mathbf{A}_{\infty}(\lambda)(x-y)}\mathbf{Q}(\lambda)\mathbf{R}_{\ell}(y), & x \geq y,\\
    \mathbf{R}_r(x)(\mathbf{I}-\mathbf{Q}(\lambda))e^{\mathbf{A}_{\infty}(\lambda)(x-y)}(\mathbf{I} - \mathbf{Q}(\lambda))\mathbf{R}_{\ell}(y), & x < y,
    \end{cases}
\end{equation}
where $\mathbf Q(\lambda): \mathbb C^4 \to \mathbb C^4$ is the projection operator  onto the stable subspace of 
$\mathbf A_\infty(\lambda)$.
If  $\|\mathbf{R}\|_{\mathbb{C}^{4 \times 4}} \in L^1(\mathbb{R}) \cap L^2(\mathbb{R})$, then
$\textsf K \in L^2(\mathbb R \times \mathbb R, \mathbb C^{4\times 4})$
and the corresponding integral operator $\mathcal K(\lambda) \in \mathcal B_2(L^2(\mathbb R, \mathbb C^4))$ is Hilbert-Schmidt~\cite{EJF}.

In the special case that 
$D=\gamma=1$,  $ \Psi(x) = \operatorname{sech}(x)$ solves \eqref{cq3} with
$\alpha = -\frac 12$. Let $\mu(\lambda) = \sqrt{1 - 2i \lambda}$ and $\nu(\lambda) = \sqrt{1 + 2i \lambda}$.
Adapting a formula from~\cite{Kap},  the Evans function for the hyperbolic secant pulse
is given by~\cite{Gallo}
\begin{equation}\label{eq:Evans}
E(\lambda) \,\,=\,\, \frac{16 \lambda^4}{ ( 1 + \mu(\lambda))^4 ( 1 + \nu(\lambda))^4}.
\end{equation}
Furthermore, a calculation~\cite{Gallo} shows that 
\begin{equation}
   \textsf{K}(x,y;\lambda) = 
   \begin{cases} 
    \operatorname{sech}(x) \operatorname{sech}(y) \textsf L(y-x;\mu(\lambda),\nu(\lambda)), & x \geq y,\\
    \operatorname{sech}(x) \operatorname{sech}(y) \textsf L(x-y;\mu(\lambda),\nu(\lambda)), & x < y,    
    \end{cases}
\end{equation}
where
\begin{equation}
 \textsf L(w;\mu,\nu) \,\,=\,\, 
 \begin{bmatrix} -6 \left( \frac{e^{\mu w}}{\mu} + \frac{ e^{\nu w}}{\nu}\right) & -2\sqrt{3}i \left( \frac{e^{\mu w}}{\mu} - \frac{e^{\nu w}}{\nu} \right)\\
        2 \sqrt{3}i \left( \frac{e^{\mu w}}{\mu} - \frac{e^{\nu w}}{\nu} \right) & 
        -2 \left( \frac{e^{\mu w}}{\mu} + \frac{e^{\nu w}}{\nu} \right)
    \end{bmatrix}.
\end{equation}
In particular, $\textsf K$ is Lipschitz continuous and $\textsf K$ and its partial derivatives
 decay exponentially. Consequently, by a theorem of Zweck et al~\cite{zweck2024regularity},
the Hilbert-Schmidt operator, $\mathcal K(\lambda)$, is also trace class, 
 $\mathcal K(\lambda)\in {\mathcal B}_1(L^2(\mathbb R, \mathbb C^4))$. 

By Theorem~\ref{truncthm} and Theorem~\ref{thmcompbound} with $r=1$, 
the overall error between
the regular Fredholm determinant of $\mathcal K(\lambda)$ and its numerical
approximation on the truncated domain $[-L,L]$ with grid spacing $\Delta x$ is 
\begin{equation}
| \operatorname{det}_1(\mathcal I + \mathcal K(\lambda)) \,\,-\,\, d_{1,Q} |
\,\,\leq\,\, e^{-aL} \boldsymbol\Phi\left(\frac{8C}a\right) \,\,+\,\, \frac{e}4 \sqrt{\frac{\pi}2}
\boldsymbol\Phi(16L \| \textsf K\left(\lambda)\|_{1}\right)\Delta x,
\end{equation}
where $a$ and $C$ are given in \eqref{truncexpass}.

\begin{figure}[!thb]
    \centering
    \includegraphics[width=0.48\linewidth]{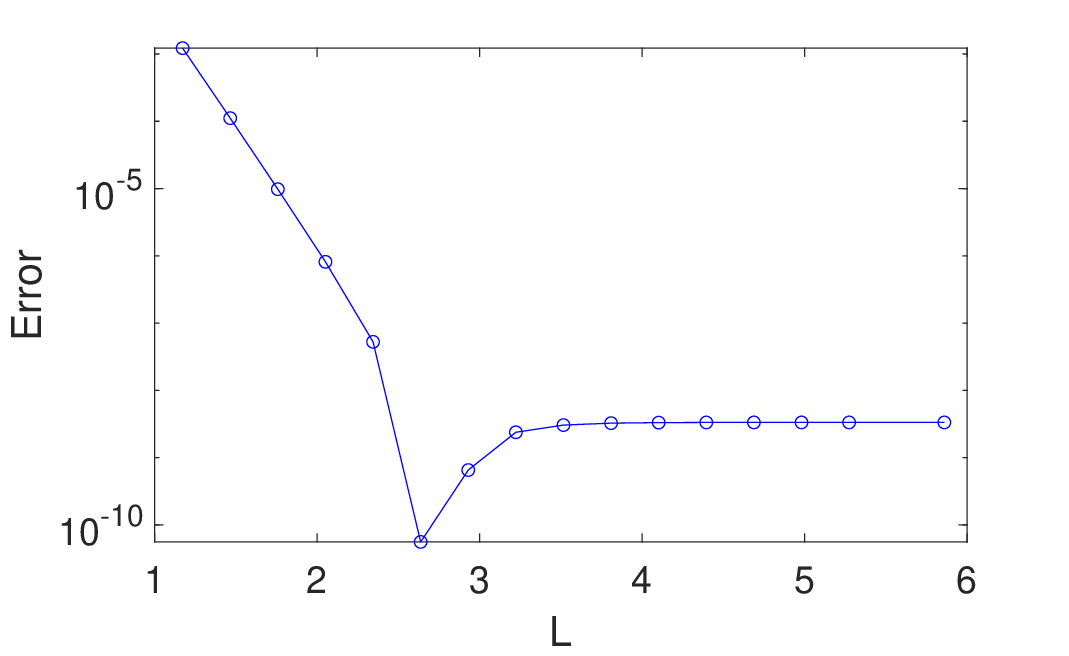}
        \includegraphics[width=0.48\linewidth]{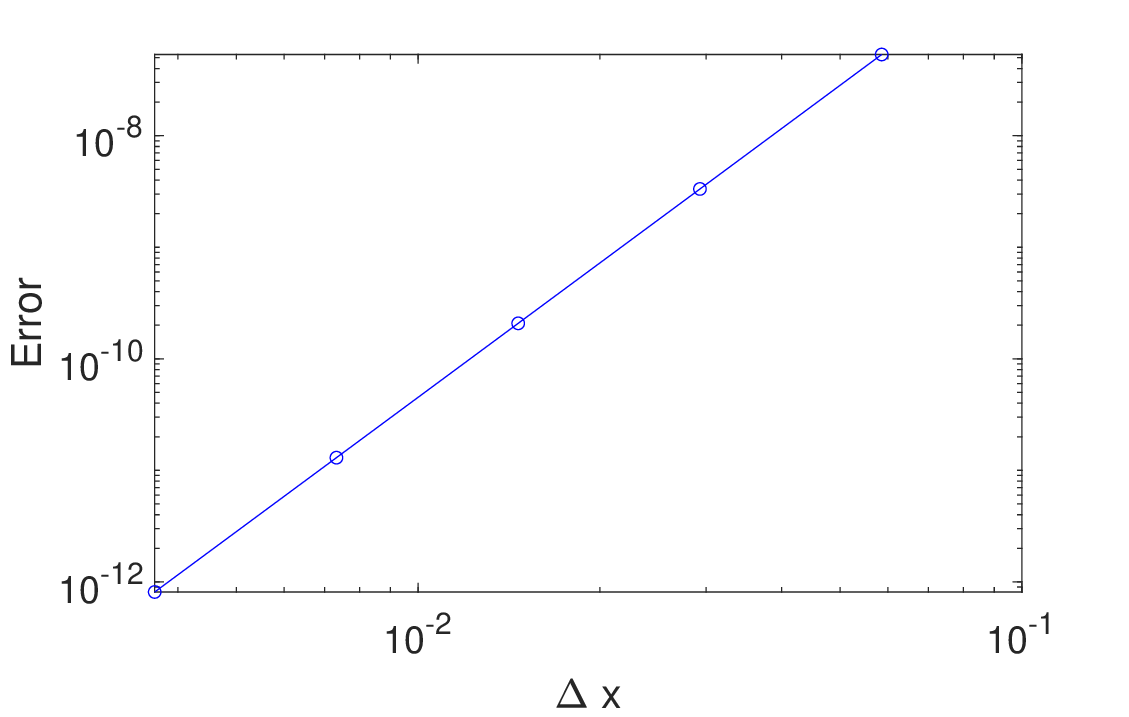}
    \caption{Error in the numerical approximation of the regular Fredholm determinant 
    at $\lambda = 0$. Left: Error as a function of the half-width, $L$, of the truncated domain for constant grid spacing, $\Delta x = 0.0293$.
    Right: Error  as a function of  $\Delta x$ with $L\approx 7.3$.}
    \label{fig:LConv}
\end{figure}

In the left panel of Figure~\ref{fig:LConv}, we plot the error between $\det_1(\mathcal{I} + \mathcal{K}(0))$ and its approximation as a function of  $L$ with a constant spacing of $\Delta x = 0.0293$. 
As expected, the error decays exponentially  as $L$ increases from $1$ to about $2.8$,
in accord with the domain truncation error estimate in Theorem~\ref{truncthm}. 
However, for larger $L$ the error  is dominated by the quadrature error,
which is independent of $L$ since we chose $\Delta x$ to be constant. 
We note that when $L=2.8$ the amplitude of the sech pulse is about 12\% of its maximum
value. This result suggests that to accurately compute the Fredholm determinant
we can used a  truncated window that is significantly less than would be required to
numerically compute the pulse solution of the nonlinear wave equation.

In the right panel of Figure~\ref{fig:LConv}, we plot the error between $\det_1(\mathcal{I} + \mathcal{K}(0))$ and its approximation as a function of  $\Delta x$ using a truncated domain
with a fixed half-width of $L \approx 7.32$.
In particular, with  $\Delta x \approx 0.0037$, the error  is less than $10^{-12}$. 
Since $\textsf K$ is Lipschitz continuous, by Theorem \ref{thmcompbound} we know that
the quadrature error should decay, at worst, as $\mathcal O(\Delta x)$. However,  we see that the actual error decays as $\mathcal O(\Delta x^4)$, which is the maximum level of accuracy achievable with the 
composite Simpson's rule. 
Given that $\textsf K$ is $C^{\infty}$ off the diagonal, this result is not surprising.
In the left panel of Figure \ref{fig:Det1AtZero5DeltaX}, we compare results obtained with the  analytical formula~\eqref{eq:Evans} for the Evans function to those with
the numerical regular  Fredholm determinant
as $\lambda \to 0$ along the imaginary axis. The  numerical determinant gains approximately one order of magnitude in accuracy each time $\Delta x$ is halved. 

\begin{figure}[h]
    \centering
    \includegraphics[width=0.48\linewidth]{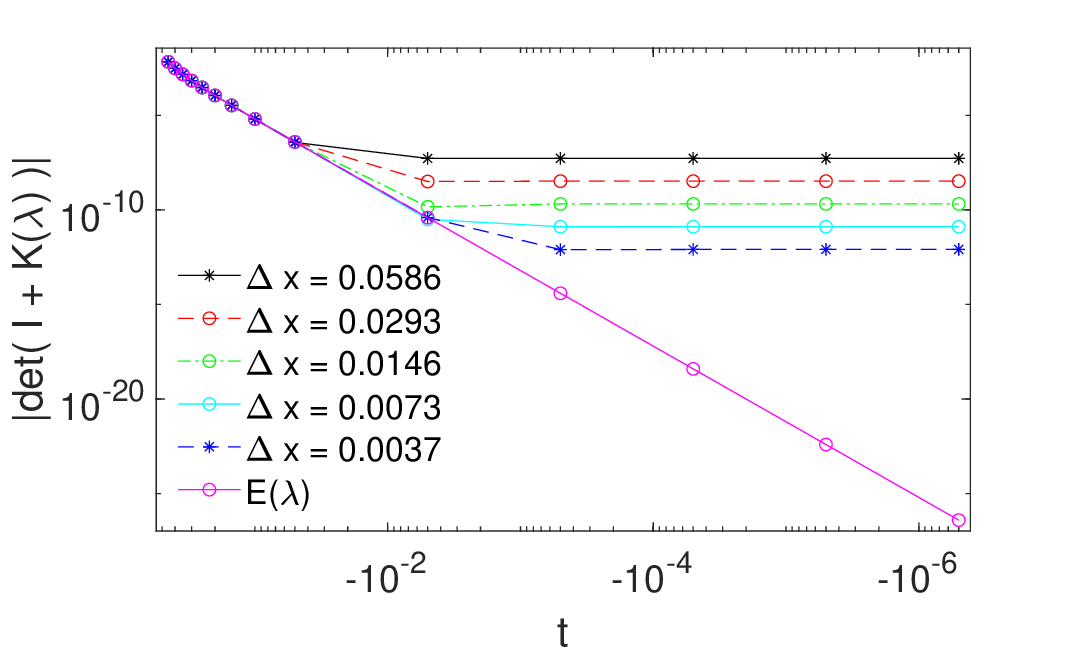}
    \includegraphics[width=0.51\linewidth]{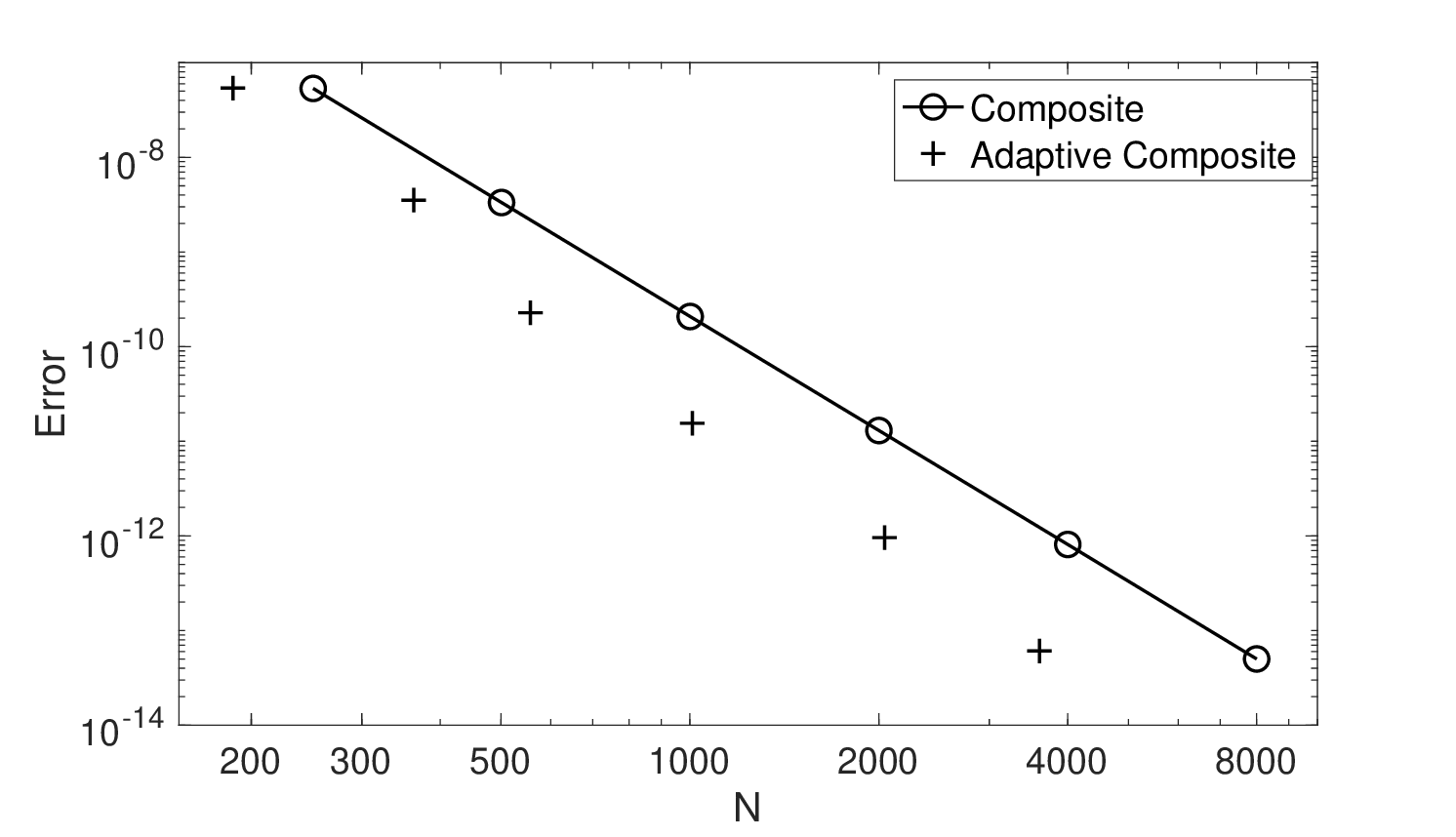} 
    \caption{Left: Dependence of the approximate determinant, $|\det_{1,Q}(\mathcal{I} + \mathcal{K}(\lambda))|$, on the quadrature spacing $\Delta x$ as $\lambda \to 0$ along the line $\lambda(t) = it$. Right: Comparison of error in approximate determinant for constant step sizes (circles) and adaptively chosen step sizes (plus signs)  as a function of the number of points in the discretization of the interval, with $L\approx 7.3$.}
    \label{fig:Det1AtZero5DeltaX}
\end{figure}

To investigate the advantage of using an adaptive quadrature rule,
in the right panel of Figure~\ref{fig:Det1AtZero5DeltaX} 
we plot the error in the approximate determinant
at $\lambda=0$ as a function of the number of points, $N$, in the discretization of the interval. Results with constant step sizes are shown with circles while those obtained with  adaptively chosen step sizes are shown with plus signs. For the adaptive method the step
sizes varied by a factor of 8 with larger step sizes towards the end of the interval where the 
amplitude of the pulse was smaller. To achieve an error of $10^{-13}$ with the adaptive method we need less than half the number of points as with constant step sizes. Since the computational time is dominated by the time to form the $4N\times 4N$ matrix, $\mathbf K_Q$, the adaptive method
is four times as efficient.

\section*{Statements and Declarations}

This work was funded by the National Science Foundation under DMS-2106203, DMS-2525548 and DMS-2106157.
The computational code and simulation data sets used for the numerical results in this paper  are available from the corresponding author on reasonable request. 
The authors have no competing interests to declare that are relevant to the content of this article.

\bibliographystyle{siamplain}
\bibliography{Fredholm}

\end{document}